\documentclass{article}
\usepackage{subcaption}
\usepackage{graphicx, xcolor} 
\usepackage{amsmath,amssymb, comment, url}
\usepackage{amsthm}
\usepackage{float} 
\usepackage[
backend=biber,
style=numeric,
sorting=none
]{biblatex}
\theoremstyle{definition}
\newtheorem{definition}{Definition}[section]
\newtheorem{theorem}{Theorem}[section]
\newtheorem{proposition}[theorem]{Proposition}
\newtheorem{lemma}[theorem]{Lemma}
\newtheorem{remark}[theorem]{Remark}
\newtheorem{assumption}[theorem]{Assumption}
\title{Characterizing resonant solitary states as relative equilibria via normal form reduction}

\author{
Bengi D\"onmez\thanks{
Department of Mathematics, Vrije Universiteit Amsterdam, Amsterdam, the Netherlands,
\texttt{b.donmez@vu.nl}}
\quad
Bob Rink\thanks{
Department of Mathematics, Vrije Universiteit Amsterdam, Amsterdam, the Netherlands,
\texttt{b.w.rink@vu.nl}}
}
\date{}
\begin{document}

\maketitle
\begin{abstract}

We investigate solitary states in coupled oscillator networks, where a single oscillator detaches from a synchronized cluster and evolves at a distinct mean frequency. We focus on resonant solitary states, in which the detached oscillator’s mean frequency is shaped by resonance with a mode of the cluster.
Here, we introduce a normal-form and symmetry-based approach to identify such states. We reduce the high-dimensional network dynamics to a two-dimensional second-order system that couples a rotator representing the solitary oscillator to a harmonic oscillator representing the dominant Laplacian mode of the cluster. We use normal-form transformations to introduce symmetry into this system. Using this symmetry, we reformulate the frequency selection problem as the search for relative equilibria of the normal form symmetry. We derive explicit algebraic conditions for the existence of these relative equilibria, and hence for resonant solitary states and their effective frequencies.
\end{abstract}
\section{Introduction}
Networks of coupled oscillators arise abundantly in physics, biology, and engineering, and Kuramoto-type models are canonical examples for studying their synchronization and desynchronization. An important dynamical phenomenon observed in such networks is the existence of solitary states~\cite{maistrenko_solitary_2014,jaros_chimera_2015}. Solitary states occur when one or a small subset of oscillators desynchronizes from the rest of the network, which remains synchronized.

The study of these states is motivated by the dynamics of modern power grids. In systems with physical inertia, solitary states can coexist with stable synchrony, creating a regime of multistability \cite{jaros_solitary_2018} where a localized perturbation can kick a single “troublemaker” node out of sync without collapsing the entire synchronized cluster~\cite{nitzbon_deciphering_2017,menck_how_2014,hellmann_network-induced_2020}. This coexistence implies that a single node can detach from the synchronized cluster and evolve with a distinct mean frequency, while the rest of the network remains synchronized. A solitary node can often support multiple stable solitary frequencies~\cite{niehues_resonant_2024,hellmann_network-induced_2020,halekotte_transient_2021}.

In this paper, we focus on a specific type of solitary state, the resonant solitary state, in which the frequency of a solitary oscillator resonates with a specific eigenmode of the synchronized cluster. The solitary oscillator can exchange energy with the synchronized cluster through resonant network modes, so this phenomenon is a product of the underlying network topology~\cite{menck_how_2014,nitzbon_deciphering_2017,zhang_vulnerability_2020,niehues_resonant_2024}. Its long-term stability and frequency are determined by how effectively its periodic forcing excites the network modes \cite{zhang_vulnerability_2020,niehues_resonant_2024}. We investigate resonant solitary states from a dynamical systems perspective. Specifically, we study second-order Kuramoto oscillators on a network and consider a minimal solitary-node configuration, where a synchronized cluster is weakly coupled to a single solitary oscillator through a single attachment node.

Our work is heavily inspired by \cite{niehues_resonant_2024}, which developed a self-consistency method based on partial linearization and averaging to identify resonant solitary frequencies in complex networks. However, that approach lacks a precise dynamical-systems characterization of these states; in particular, it does not yield rigorous conditions for the existence and stability of solutions. In contrast, this paper proves that the dynamics of the Kuramoto model with inertia can be transformed into an $ S^1$-equivariant system, in which resonant solitary states correspond to relative equilibria. We derive explicit algebraic conditions for their existence and obtain a stability criterion via transverse linearization.

A major advantage of our approach is its generality. The same argument and computation apply to weakly perturbed oscillator-rotator systems whose equations of motion take the form \begin{equation}\label{eq:main} \begin{cases} \ddot{\eta}+\omega_0^2\eta = \varepsilon F_1(\eta,\dot{\eta},\psi,\dot{\psi}) + \mathcal{O}(\varepsilon^2),\\[1mm] \dot{\psi} = \Omega_0+\varepsilon F_2(\eta,\dot{\eta},\psi,\dot{\psi})+ \mathcal{O}(\varepsilon^2), \end{cases} \end{equation}
with $0<\varepsilon\ll1$. In the above setting, $\eta$ denotes the amplitude of a network mode, while $\psi$ is the phase of the solitary oscillator. A near-identity normal-form transformation brings this system into a form of which the dynamics is $ S^1$-equivariant. More precisely, the resonant normal form preserves the $S^1$ symmetry generated by the unperturbed dynamics. Periodic solutions in this setting can then be sought as relative equilibria of the normal-form symmetry. Hence, we convert the search for periodic solutions into solving explicit algebraic equations. Resonant solitary states in the Kuramoto model with inertia provide a concrete, physically motivated application.

\begin{figure}[H] \centering \begin{subfigure}[b]{0.79\textwidth} \centering \includegraphics[width=\linewidth]{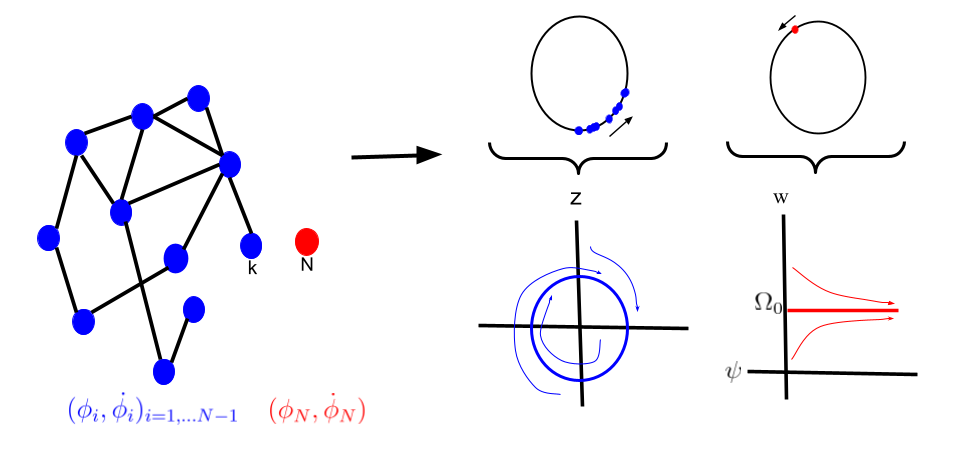} \caption{Decoupled regime.} \label{fig:decoupled-regime} \end{subfigure} \hfill \begin{subfigure}[b]{0.79\textwidth} \centering \includegraphics[width=\linewidth]{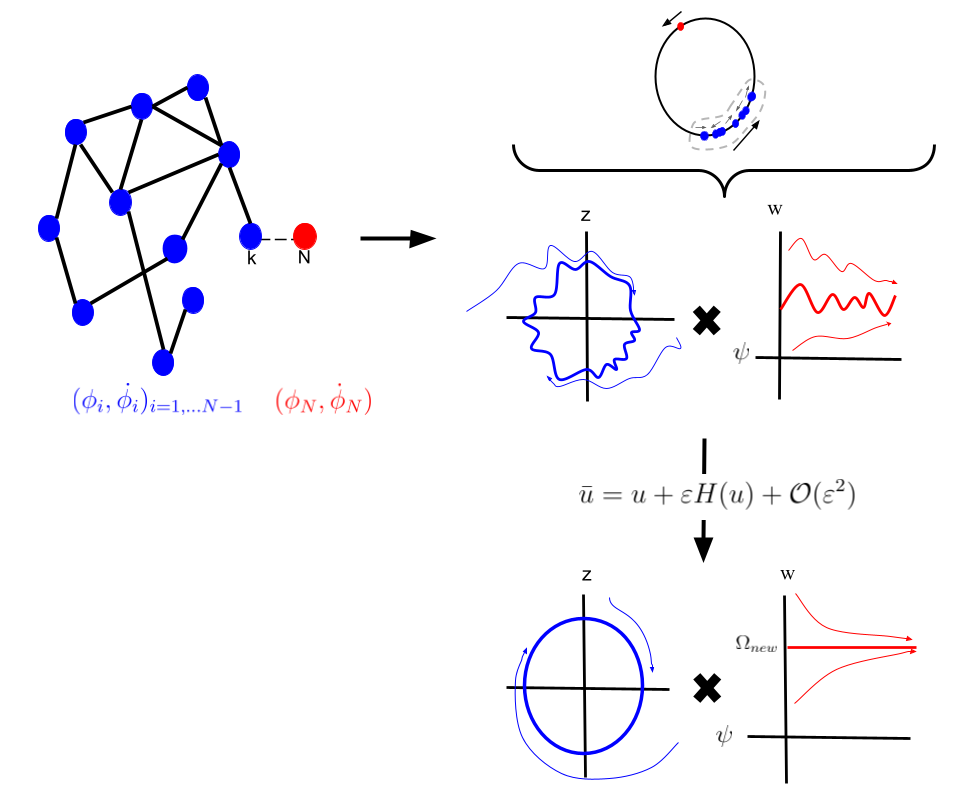} \caption{Weakly coupled regime and normal-form transformation.} \label{fig:coupled-regime} \end{subfigure} 
\caption{Schematic of the reduction from the solitary-node network to an oscillator--rotator system and its resonant normal form. We use the reduced coordinates $u=(z,\psi,w)^\top$, where $z\in\mathbb{C}$ represents the selected cluster mode, $\psi\in S^1$ is the solitary phase relative to the synchronized cluster, and $w\in\mathbb{R}$ describes the correction to the solitary angular velocity. (a) Decoupled regime. (b) Weakly coupled regime and normal-form transformation.}
\label{fig:decoupled-coupled-reduction} \end{figure}

Figure~\ref{fig:decoupled-coupled-reduction} illustrates the methods used in our analysis. Starting from the network on the left, we first reduce the dynamics near the partially synchronized state to one cluster mode coupled to a solitary oscillator. In the decoupled limit, shown in panel~(a), the selected cluster mode is a harmonic oscillator to leading order, while the solitary oscillator approaches uniform rotation. In an appropriate coordinate, the harmonic oscillator therefore traces a circle, whereas the solitary angular velocity approaches a limiting constant value $\Omega_0$. When there is weak coupling, the solitary oscillator periodically forces the selected cluster mode. The upper part of panel~(b) represents the resulting weakly perturbed oscillator-rotator dynamics. The circular motion of $z$ and the uniform solitary rotation are perturbed by the coupling. We then apply the normal-form transformation indicated by the vertical arrow. In the normal-form coordinates, the resonant terms are separated from the nonresonant perturbations. A resonant solitary state is represented by uniform rotation of $z$ together with a constant value of $w$, so that the solitary oscillator rotates with the constant angular velocity $\Omega_0+\varepsilon w$. Sections~5--7 show that this motion is a relative equilibrium of the $S^1$-equivariant normal form.

\subsubsection*{Organization of the paper}

The remainder of this paper is organized as follows. In Section~2, we introduce the Kuramoto model with inertia, describe the solitary-node configuration, and state the scaling assumptions used throughout the paper. In Section~3, we reduce the scaled network dynamics to the four-dimensional oscillator--rotator system \eqref{eq:u} by expanding around the synchronized cluster state, projecting onto Laplacian eigenmodes, and retaining a dominant resonant mode. In Section~4, we construct the first-order resonant normal form of this reduced system using near-identity transformations and a Fourier--Taylor decomposition of the equations of motion. In Section~5, we identify the $S^1$-symmetry of the normal form of system \eqref{eq:main} generated by its unperturbed flow, and prove that the truncated normal form is equivariant under this symmetry. In Section~6, we characterize resonant solitary states as relative equilibria and reduce the relative-equilibrium problem to algebraic conditions on a cross-section. In Section~7, we specialize these conditions in the near-resonant normal form of the solitary state problem and compute explicit relative equilibria together with their existence and stability conditions. Section~8 contains a conclusion and discussion. Appendix A gives the technical details of the reduction from the full $2N$-dimensional Kuramoto system to the four-dimensional oscillator-rotator system \eqref{eq:u}.

Readers primarily interested in the normal-form and symmetry analysis may proceed directly to Sections~4-6 after consulting Section~3 for the reduced system \eqref{eq:u}.

\section{Model Setup}
This section introduces the solitary oscillator problem, which will serve as the main application of the general method developed later. We begin with the Kuramoto model with inertia, describe the solitary-node configuration, and choose a scaling regime in which the internal cluster dynamics remains of order one, while damping, solitary-node input, and solitary-cluster coupling enter as a perturbation. This formulation identifies both the small parameter and the reference synchronized cluster state around which the reduction in Section~3 is carried out. The output of this section is the re-scaled model \eqref{eq:ScaledModel}, which provides the starting point for the oscillator-rotator reduction in Section~3. 
\begin{figure}[ht]
\centering
\includegraphics[width=0.35\textwidth]{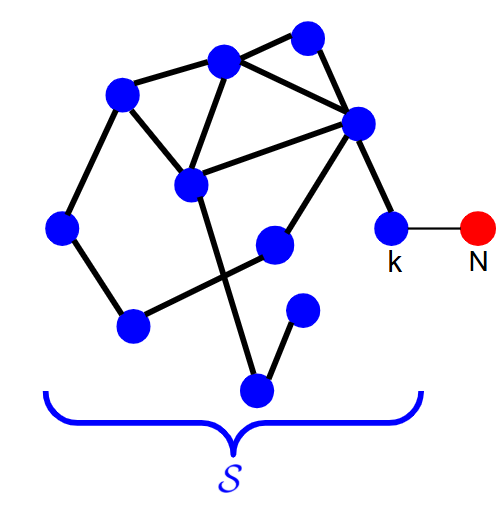}
\caption{A cluster of oscillators $S$ connected to a solitary oscillator $N$. The blue nodes form the connected cluster
$S={1,\ldots,N-1}$, while the red node $N$ is the solitary oscillator. The solitary oscillator is coupled only to the attachment node $k$ through the weak link $K_{kN}=K_{Nk}$. The black edges represent the internal couplings within the cluster.}
\label{fig:solitary-node-configuration}
\end{figure}
We consider a Kuramoto model with inertia with $N$ oscillators, where nodes $i\in \mathcal S :=\{1,\dots,N-1\}$ form a large, connected cluster and node $N$ is a solitary oscillator coupled only to a single cluster node $k$, which we call the attachment node. This is illustrated in Figure~\ref{fig:solitary-node-configuration}. The dynamics of this system is given by
\begin{align}\label{eq:InitialSystem}
&m\,\ddot{\phi}_i - P_i + \alpha\,\dot{\phi}_i
= \sum_{j=1}^{N-1} K_{ij}\,\sin(\phi_j-\phi_i)
+ \delta_{ik}\,K_{kN}\,\sin(\phi_N-\phi_k),
\quad i\in \mathcal S,\\ \nonumber
&m\,\ddot{\phi}_N - P_N + \alpha\,\dot{\phi}_N
= K_{Nk}\,\sin(\phi_k-\phi_N).
\end{align}
Here $\phi_i$ denotes the phase of the $i$th oscillator, $m>0$ denotes the common inertia, $\alpha>0$ the common damping, $P_i$ the net input (power) at node $i$, $K_{ij}=K_{ji}\geq 0$ the internal coupling strengths within the cluster, and $K_{Nk}=K_{kN}\geq 0$ the weak link between the solitary oscillator and the attachment node $k$. Moreover, $\delta_{ik}$ denotes the Kronecker delta, that is,
\begin{equation*}
\delta_{ik} =
\begin{cases}
1, & i=k,\\
0, & i\neq k.
\end{cases}
\end{equation*}
For $K_{kN}=0$, that is, when there is no connection between the solitary oscillator $N$ and the attachment node $k$, the solitary oscillator is decoupled from the cluster. We call this the {\it decoupled regime.} In this regime, we assume that the subsystem involving nodes $i \in \mathcal S$ has a stable frequency-synchronized trajectory
\begin{equation*}
\phi_i(t)=\phi_i^*+\Omega t, \qquad i\in \mathcal S.
\end{equation*}
The associated synchrony manifold is given by
\begin{equation*}\label{eq:syncManifold}
\mathcal{M}_{\mathrm{syn}}
=
\left\{
( \phi, \dot \phi) \in \mathbb{T}^{N} \times  \mathbb R^{N}
:\;
\dot \phi_i=\dot \phi_j
\text{ for all } i,j\in \mathcal S
\right\},
\end{equation*}
and we take this manifold as the reference invariant set for our reduction.

The solitary oscillator $N$ can have a different intrinsic frequency from the cluster. When $K_{kN}$ is small, the forcing by the solitary oscillator
\begin{equation*}
\delta_{ik}K_{kN}\sin(\phi_N-\phi_i)
\end{equation*}
acts as a perturbation on the synchrony manifold. Depending on the choice of parameters, this manifold may either lose stability and break down or persist as a slightly deformed invariant set. We are interested in the regime in which the synchrony manifold persists under weak coupling to the solitary node, and in particular in how this depends on the mean frequency of the solitary oscillator.

We therefore consider a scaling regime in which the internal cluster dynamics remain of order one, while damping, the solitary oscillator input, and the coupling between the solitary node and the cluster enter as small perturbations. More precisely, we assume that
\begin{equation*}\label{eq:scaling}
\hat P_i: =\frac{P_i}{m}=\mathcal{O}(1), \qquad  \kappa_{ij}:=\frac{K_{ij}}{m}=\mathcal{O}(1),
\quad i,j\in \mathcal S,
\end{equation*}
whereas
\begin{equation*}
\frac{\alpha}{m} = \varepsilon\,\hat\alpha,\qquad
\frac{P_N}{m} = \varepsilon\,\hat P_N,\qquad
\frac{K_{Nk}}{m} = \varepsilon^2\,\kappa_{Nk},
\qquad 0<\varepsilon\ll 1,
\end{equation*}
with $\hat\alpha,\hat P_N,\kappa_{Nk}=\mathcal{O}(1)$. Then \eqref{eq:InitialSystem} can be written as
\begin{align}\label{eq:ScaledModel}
\ddot{\phi}_i
&= \hat{P_i}
+ \sum_{j=1}^{N-1}\kappa_{ij}\,\sin(\phi_j-\phi_i)
- \varepsilon\,\hat\alpha\,\dot{\phi}_i
+ \varepsilon^2\,\delta_{ik}\,\kappa_{kN}\,\sin(\phi_N-\phi_i),
\quad i\in \mathcal S,\\ \nonumber
\ddot{\phi}_N
&= \varepsilon\,\hat P_N
- \varepsilon\,\hat\alpha\,\dot{\phi}_N
+ \varepsilon^2\,\kappa_{Nk}\,\sin(\phi_k-\phi_N).
\end{align}
Observe that the forcing on the cluster caused by the solitary node is only of size $\mathcal{O}(\varepsilon^2)$ compared to the $\mathcal{O}(1)$ collective coupling within the cluster. The above scaling is chosen so that the solitary oscillator still rotates with an order one mean frequency in the decoupled limit. Indeed, if $K_{Nk}=0$, then the solitary equation in \eqref{eq:ScaledModel} reduces to
\begin{equation} \label{eq:uncoupledSolitaryDynmaics}
\ddot{\phi}_N
=
\varepsilon \hat P_N
-
\varepsilon \hat\alpha\,\dot{\phi}_N,
\end{equation}
whose attracting uniform rotation satisfies
\begin{equation}\label{eq:uncoupledSolitaryFrequency}
\dot{\phi}_N \to \frac{\hat P_N}{\hat\alpha} \quad \text{as} \quad t \to \infty.
\end{equation}
This scaling in \eqref{eq:uncoupledSolitaryDynmaics} is not arbitrary. By taking both the damping and the power input of the solitary oscillator to be of order $\varepsilon$, the decoupled solitary oscillator has an order one mean angular velocity. As a consequence, the term
\begin{equation*}
\varepsilon^2\,\delta_{ik}\,\kappa_{kN}\sin(\phi_N-\phi_i)
\end{equation*}
acts on the cluster as a weak time-periodic forcing whose frequency is determined by the mean rotational frequency of the solitary oscillator. It produces a forcing of small amplitude but order one frequency, which can resonate with an eigenmode of the linearized cluster dynamics. In that resonant regime, the response of the corresponding network mode may become order one, so that the solitary oscillator can induce a persistent deformation of the synchrony manifold and generate a form of period matching between the cluster and the solitary node despite the small coupling $K_{Nk}$.

\section{Reduction to an Oscillator-Rotator Model}

This section connects the original network model to the normal-form analysis done   later. Starting from the re-scaled Kuramoto system \eqref{eq:ScaledModel}, we reduce the dynamics near the synchronized cluster state to a four-dimensional system with one harmonic oscillator and one rotator. The harmonic oscillator describes the dominant Laplacian mode of the cluster, and the rotator describes the solitary oscillator relative to the cluster. This reduced system is the system to which we apply the normal-form transformation in Section~4.

The reduction has four steps. First, we move to a co-rotating frame and perturb around the synchronized cluster state. Second, we rewrite the cluster dynamics in terms of Laplacian eigenmodes. Third, we project onto the mode that is expected to dominate near resonance with the solitary oscillator. Finally, we write the resulting second-order equations as a first-order system and introduce a complex coordinate for the cluster mode. The result is the compact system \eqref{eq:u}, which has the form needed for the normal-form and symmetry analysis. The full derivation of the reduction from the original network equations to \eqref{eq:u} is given in Appendix~A; here we only outline the main steps and introduce the reduced variables that will be used in the subsequent analysis.

\subsection{Perturbation around the synchronized cluster state in a co-rotating frame.}\label{subsec:newcoordinates}

The purpose of this step is to separate the collective rotation of the synchronized cluster from the deviations caused by the solitary oscillator.
We first expand the cluster phases around the frequency-synchronized solution by writing 
\begin{equation}\label{var:clusterperturbed}
\phi_i(t)=\phi_i^*+\Omega t+\varepsilon\,\nu_i(t), \qquad i\in \mathcal S,
\end{equation}
and introduce a co-rotating variable for the solitary oscillator:
\begin{equation}\label{var:solitarycorotating}
\psi_N(t):=\phi_N(t)-\phi_k^*-\Omega t.
\end{equation}
Substituting these expressions into \eqref{eq:ScaledModel} and expanding to $\varepsilon$ around the synchronized cluster state yields weakly coupled equations for the cluster perturbation $\nu$ and solitary variable $\psi_N$. See Appendix \ref{approximation_near_solitary_state} for the details.

\subsection{Projection onto Laplacian eigenmodes and localization in a dominant mode.}

The aim of this subsection is to identify which cluster mode is actually driven by the solitary oscillator. After linearization around the synchronized cluster state, the cluster response is naturally decomposed into Laplacian eigenmodes. This is the point where the network topology enters the reduced dynamics: the solitary oscillator can only resonantly force modes that are visible at the attachment node. The dominant-mode assumption then selects the single mode that will become the harmonic-oscillator component of the reduced oscillator-rotator system.

The linearized cluster dynamics is governed by a weighted graph Laplacian $L$, which is orthogonally diagonalizable:
\begin{equation*}
L = Q\Lambda Q^\top,
\end{equation*}
where $\Lambda=\mathrm{diag}(\lambda_1,\dots,\lambda_{N-1})$ and $Q\in\mathbb{R}^{(N-1)\times(N-1)}$ is an orthogonal matrix. We introduce the Laplacian-mode coordinates
\begin{equation*}
\eta := Q^\top \nu.
\end{equation*} We also write
\begin{equation*}
q^{(j)} := Q^\top e_j.
\end{equation*}
Thus, $q^{(j)}$ records the values of all Laplacian eigenvectors at node $j$. We then fix an index $\ell\in\{1,\dots,N-1\}$ such that
\begin{equation}\label{eq:assumptionq}
q^{(k)}_\ell \neq 0.
\end{equation}
We assume that the cluster response is dominated by this mode, in the sense that
\begin{equation*}
\nu(t)=Q\eta(t)=\eta_\ell(t)Qe_\ell+\mathcal{O}(\varepsilon).
\end{equation*}
Thus, to leading order, the perturbation of the synchronized cluster is assumed to be aligned with a single Laplacian eigenmode. This assumption is motivated by resonance. The solitary oscillator acts on the cluster as a weak periodic forcing with leading frequency $ \frac{\hat P_N}{\hat\alpha}$. Among the Laplacian modes of the linearized cluster, the mode whose natural frequency is closest to this forcing frequency gives the dominant response, provided that it is actually excited through the attachment node. Therefore, we choose an index $\ell$ that satisfies \eqref{eq:assumptionq} and whose
corresponding natural frequency $\sqrt{\lambda_\ell}$ is close to the
solitary forcing frequency. The remaining modes are assumed to be non-resonant and hence contribute only higher-order corrections. We write
\begin{equation*}
\lambda_\ell=\omega_0^2,
\end{equation*}
i.e., $\omega_0$ is the natural frequency of the dominant resonant cluster mode.

At this stage, the high-dimensional network dynamics has been reduced to two coupled second-order equations: one equation for the amplitude of a single resonant cluster mode and one equation for the phase of the solitary oscillator relative to the synchronized cluster. Equivalently, after introducing the corresponding velocities, the reduced dynamics is four-dimensional. The selected cluster mode behaves as a weakly forced harmonic oscillator, while the solitary phase provides the rotator component of the reduced oscillator-rotator system.

\subsection{Reduced second-order system.}
Under the dominant-mode assumption, the full cluster dynamics reduces to a single cluster mode $\eta_\ell$, coupled to the solitary phase $\psi_N$. From now on, to simplify notation, we drop the indices $\ell$ and $N$ and write
\begin{equation*}
\eta:=\eta_\ell,\qquad
\psi:=\psi_N,\qquad
q:=q^{(k)}_\ell,\qquad
\kappa:=\kappa_{Nk}=\kappa_{kN}.
\end{equation*}
With this convention, our equations of motion become
\begin{align} \label{eq:ScaledModelReduced}
\ddot{\eta} + \omega_0^2\,\eta
&= \varepsilon(
-\hat\alpha\,\dot{\eta}
+ C_{\ell}\,\eta^2
+ \kappa\,q\,\sin\psi
)
+ \mathcal{O}(\varepsilon^2),\\ \nonumber
\ddot{\psi}
&= \varepsilon\,(\hat P_N-\hat\alpha\,\Omega
-\,\hat\alpha\,\dot{\psi})
+ \mathcal{O}(\varepsilon^2).
\end{align}
Note that the coefficient $C_{\ell}$ comes from projecting onto the $\ell$-th Laplacian mode the quadratic nonlinear remainder obtained by Taylor expanding the internal coupling terms inside the synchronized cluster. The details of this calculation are given in Appendix~A.

Thus, \eqref{eq:ScaledModel} reduces to a harmonic oscillator for the dominant cluster mode, weakly coupled to a forced rotator describing the solitary oscillator.

\subsection{Reduced first-order system.}
We rewrite the second-order system \eqref{eq:ScaledModelReduced} as a first-order vector field to apply the normal-form transformation in Section~4. For this purpose, we introduce
\begin{equation}\label{eq:coordinates}
v:=\dot{\eta},\qquad
w:=\dot{\psi}.
\end{equation}
With new coordinates \eqref{eq:coordinates}, the reduced system \eqref{eq:ScaledModelReduced} becomes
\begin{align}\label{eq:vector-splitting-no-x}
\begin{pmatrix}
\dot{\eta}\\[2pt]
\dot{v}\\[2pt]
\dot{\psi}\\[2pt]
\dot{w}
\end{pmatrix}
&=
\begin{pmatrix}
v\\
-\omega_0^2 \eta\\
w\\
0
\end{pmatrix}
\;+\;
\varepsilon\,
\begin{pmatrix}
0\\[2pt]
-\hat\alpha\,v + C_{\ell}\,\eta^2 + \kappa\,q\,\sin\psi\\[2pt]
0\\[2pt]
\hat P_N - \hat\alpha\,\Omega - \hat\alpha\,w
\end{pmatrix}\\ \nonumber
&\quad+\;
\varepsilon^2\,
\begin{pmatrix}
0\\[2pt]
-\kappa\,q^2\,\eta\,\cos\psi\\[2pt]
0\\[2pt]
-\kappa\,\sin\psi
\end{pmatrix}
\;+\;
\mathcal{O}(\varepsilon^3)
\end{align}
The solitary angular velocity evolves
on the slow time scale since $\dot w=\mathcal O(\varepsilon)$. Therefore, we consider solutions for which $w$ remains
$\varepsilon$-close to its leading-order value $\Omega_0$. We thus write\begin{equation}
w=\Omega_0+\varepsilon\,\widetilde w,
\end{equation}
where $\Omega_0:=  \frac{\hat P_N}{\hat\alpha} -\Omega$, see \eqref{eq:uncoupledSolitaryFrequency} and \eqref{var:solitarycorotating}, denotes the leading-order mean angular velocity of the solitary oscillator in the co-rotating frame, and $\widetilde w$ describes its small variation. With this substitution, the equation for the solitary oscillator becomes
\begin{equation}
\dot\psi=\Omega_0+\varepsilon\,\widetilde w, \qquad \dot {\widetilde w} = \hat P_N - \hat\alpha(\Omega +\Omega_0 )+\mathcal{O}(\varepsilon).
\end{equation}
Notice that $\psi$ is a rotator with leading frequency $\Omega_0$, weakly modulated by the slow correction $\widetilde w$. To simplify notation, we henceforth write $w$ instead of $\widetilde w$.

Thus, to leading order, $\psi$ evolves with angular velocity $\Omega_0$, and the term $\kappa\,q\,\sin\psi$ in the $\dot v$-equation acts as a periodic forcing of the cluster-mode subsystem $(\eta,v)$ with forcing frequency $\Omega_0$. We are interested in how the oscillator $(\eta,v)$ responds to this forcing, and in particular in those values of $\Omega_0$ for which the response amplitude becomes large.

This occurs when the forcing frequency $\Omega_0$ is close to the frequency $\omega_0$ of the dominant cluster mode, or more generally when $\Omega_0$ is close to a rational multiple of $\omega_0$. In that case, the weak forcing generated by the solitary oscillator can accumulate over long times and produce a large-amplitude response in the dominant Laplacian mode. As a result, the dynamics of the solitary oscillator and the cluster mode become strongly correlated through phase locking or frequency matching, and the weak coupling at the attachment node $k$ can induce deformations of the synchrony manifold.

\subsection{Complex coordinates for the Cluster }
We finally introduce a complex coordinate for the selected cluster mode, $z:=\eta - \frac{i}{\omega_0}v$. This change of coordinates combines the real variables $(\eta,\dot \eta)$ into a single complex variable $z$ for the dominant Laplacian eigenmode. In the new variable $u:=(z,\psi,w)^\top$, the order one part of the system \eqref{eq:vector-splitting-no-x} is given by the vector field
\begin{equation}\label{eq:L0-vector-field}
L_0(u)\;=\;\bigl(i\omega_0 z,\ \Omega_0,\ 0\bigr).
\end{equation}

Thus the reduced dynamics is of the form
\begin{equation}\label{eq:u}
\dot{u}
=
L_0(u)
\;+\;
\varepsilon\,
F(u)
\;+\;
\varepsilon^2\,
G(u)
\;+\;
\text{higher–order terms}.
\end{equation}
Using the relations between $(\eta,v)$ and $(z,\bar z)$, we derive that the order–$\varepsilon$ term in \eqref{eq:u} reads
\begin{equation} \label{eq:ScaledReducedComplex}
F(u)
=
\begin{pmatrix}
F_{z}(u)\\[2pt]
F_{\psi}(u)\\[2pt]
F_{w}(u)
\end{pmatrix}
=
\begin{pmatrix}
-\dfrac{\hat\alpha}{2}\,(z-\bar z)
-\dfrac{i\,C_{\ell}}{4\omega_0}\,(z+\bar z)^2
- i\,\dfrac{\kappa\,q}{\omega_0}\,\sin\psi\\[6pt]
0\\[4pt]
 P - \hat\alpha\,w
\end{pmatrix},
\end{equation}
and that the order–$\varepsilon^2$ term is given by
\begin{equation*}
G(u)
=
\begin{pmatrix}
G_{z}(u)\\[2pt]
G_{\psi}(u)\\[2pt]
G_{w}(u)
\end{pmatrix}
=
\begin{pmatrix}
i\,\dfrac{\kappa\,q^2}{2\omega_0}\,(z+\bar z)\,\cos\psi\\[6pt]
0\\[4pt]
-\kappa\,\sin\psi
\end{pmatrix}.
\end{equation*}
From Section~4 onward, our analysis is no longer tied to the original network variables; it depends only on this oscillator-rotator structure and on the resonant terms selected by the normal form.
Equation \eqref{eq:u} is the main system that we will analyze in the next sections.

\section{Normal Form Transformations}\label{sec:NF-transformations}

We now analyze the reduced oscillator-rotator system obtained in Section~3. The goal of this section is to simplify the order-$\varepsilon$ part of \eqref{eq:u} by a near-identity coordinate transformation. More precisely, we seek a first-order normal form in which the order-$\varepsilon$ terms commute with the unperturbed vector field $L_0$. This commutation property will be used in Sections~5 and~6 to identify the symmetry of the normal form and to characterize resonant periodic solutions as relative equilibria. Note that the normal-form calculation below does not require the explicit network equations from which the reduced system was derived. Once the reduction in Section~3 has been carried out, the calculation depends on the reduced vector field $F$ and on the leading-order vector field $L_0$. The Kuramoto model therefore enters the calculation through the specific terms and coefficients appearing in $F$. 

Let us consider the system
\begin{equation}\label{eq:NF-system}
\dot u \;=\; L_0(u) \;+\;\varepsilon\,F(u)\;+\;\mathcal{O}(\varepsilon^2),
\quad
u=(z,\psi,w) \in X \; := \;  \mathbb{C}\times \mathbb{S}^1\times \mathbb{R}
\end{equation}
with $L_0$ as in \eqref{eq:L0-vector-field}.
In this section, we prove that there exists a near-identity coordinate transformation that puts
\eqref{eq:NF-system} into a form whose order-$\varepsilon$ term commutes with the leading-order
vector field $L_0$. The geometric effect of the normal-form transformation is illustrated in Figure~\ref{fig:coupled-regime}.
 Specifically, in the new coordinates, the system can be written as
\begin{equation}\label{eq:first-order-nf-commuting-form}
\dot u
=
L_0(u)
+
\varepsilon\,\bar F(u)
+
\mathcal O(\varepsilon^2)
\quad\text{such that}\quad
[L_0,\bar F]=0,
\end{equation}
which will be referred to as the
\emph{first-order normal form} of the system. In the next section, we analyze the truncated system
\begin{equation}\label{eq:first-order-nf-truncated}
\dot u
= \mathcal{F}(u)
:=
L_0(u)
+
\varepsilon\,\bar F(u).
\end{equation}
The commutation relation in \eqref{eq:first-order-nf-commuting-form} implies that $\mathcal{F}$ is equivariant under the action $R_\theta$ generated by $L_0$. 

Under this equivariance, periodic solutions of \eqref{eq:first-order-nf-truncated} can be interpreted as relative equilibria of the action $R_\theta$, which can be found by solving algebraic equations. To prove the main result of this section, we first make some definitions.

\begin{definition}\label{lem:lie-bracket-adjoint}
Let $H$ and $F$ be $C^1$ vector fields on $X$. Their Lie bracket is the $C^{0}$-vector field
\begin{equation}\label{eq:lie-bracket}
[H,F](u)\;:=\; D H(u) F(u)\;-\;D F(u) H(u).
\end{equation}
We denote by ${\rm{ad}}_H$ the linear operator of taking the Lie bracket with $H$, that is,
\begin{equation}\label{eq:adH-def}
\mathrm{ad}_H(F)\;:=\;[H,F]
\end{equation}
for $F$ a $C^{1}$-vector field.
\end{definition}
We now investigate how the flow of a $C^{1}$-vector field $H$ transforms an arbitrary $C^{1}$-vector field $F$. Let $s \mapsto \phi(s)$ be an integral curve of $F$, so that
$$
\frac{d\phi(s)}{ds}=F(\phi(s)),
$$
and let $\Gamma : X \to X$ be a diffeomorphism. Then the curve
$$
\psi(s):=\Gamma(\phi(s))
$$
satisfies
$$
\frac{d\psi(s)}{ds}
=
D\Gamma(\phi(s))\ F(\phi(s))
=
D\Gamma(\Gamma^{-1}(\psi(s)))\ F(\Gamma^{-1}(\psi(s))).
$$
Thus $\psi$ is an integral curve of the \emph{pushforward} vector field $\Gamma_*F : X \to \mathbb{C}\times\mathbb{R}^2$ defined by
$$
\Gamma_*F
=
(D\Gamma \ F)\circ \Gamma^{-1}.
$$
Let $\Phi^H_s$ be the time-$s$ flow of $\dot u = H(u)$. We want to compute the Taylor expansion of $(\Phi^H_s)_*F$ with respect to the flow parameter $s$.
The following lemma is a standard Lie-series formula for the transformation
of vector fields under the flow of a generator. We will omit the proof. For related results in normal-form theory, see
\cite[Theorem~9.1.5]{sanders_averaging_2007} and
\cite{murdock_normal_2003}.
\begin{lemma}\label{lem:pushforward-and-transform}
Let $F,H \in C^1$. Then the pushforward of $F$ under the time-$s$ flow of
\begin{equation*}
\dot u = H(u)
\end{equation*}
satisfies
\begin{equation}\label{eq:pushforward-ode}
\frac{d}{ds}\bigl((\Phi^H_s)_*F\bigr)
=
\mathrm{ad}_H\bigl((\Phi^H_s)_*F\bigr),
\qquad
(\Phi^H_0)_*F = F.
\end{equation}
If, moreover, $F,H \in C^2$, then formally
\begin{equation*}\label{eq:pushforward-exp-merged}
(\Phi^H_s)_*F
=
e^{s\,\mathrm{ad}_H}(F)
=
F + s\,\mathrm{ad}_H(F) + \frac{s^2}{2}\,\mathrm{ad}_H^2(F) + \mathcal{O}(s^3).
\end{equation*}
In particular, at $s=\varepsilon$,
\begin{equation*}
(\Phi_H^\varepsilon)_*F
=
e^{\varepsilon\,\mathrm{ad}_H}(F)
=
F+\varepsilon[H,F]+\mathcal{O}(\varepsilon^2).
\end{equation*}
Consequently, if
\begin{equation*}\label{eq:NF-system-merged}
\dot u
=
L_0(u)+\varepsilon F(u)+\mathcal{O}(\varepsilon^2)
\end{equation*}
and we perform the near-identity change of variables
\begin{equation}\label{eq:near-identity-flow-merged}
\bar u=\Phi_H^\varepsilon(u)=u +\varepsilon H(u) + \mathcal{O}(\varepsilon^2),
\end{equation}
then the transformed system is
\begin{equation}\label{eq:first-order-transformed}
\dot{\bar u}
=
L_0(\bar u)
+
\varepsilon\bigl(F(\bar u)+[H,L_0](\bar u)\bigr)
+
\mathcal{O}(\varepsilon^2).
\end{equation}
\end{lemma}
We will now choose the generator $H$ so that the near-identity transformation $\bar u=\Phi_H^\varepsilon(u)$ simplifies the order-$\varepsilon$ term in \eqref{eq:first-order-transformed}. By Lemma~\ref{lem:pushforward-and-transform}, the transformed order-$\varepsilon$ term is
\begin{equation*}
F+[H,L_0] = F - \operatorname{ad}_{L_0}H .
\end{equation*}
This motivates us to study linear operator 
\begin{equation}\label{eq:homological-operator}
\operatorname{ad}_{L_0}:\mathfrak X(X)\to \mathfrak X(X),
\qquad
\operatorname{ad}_{L_0}H := [H,L_0],
\end{equation}
where $\mathfrak X(X)$ denotes the space of vector fields on $X$. We call
$\operatorname{ad}_{L_0}$ the homological operator associated to $L_0$. Terms in the image of
$\operatorname{ad}_{L_0}$ can be removed from $F$ by a suitable choice of the generator
$H$. Therefore, we analyze the eigenvectors and eigenvalues of
$\operatorname{ad}_{L_0}$. We do this using on a Fourier-Taylor basis.

\begin{lemma}
\label{lem:homological-eigenpairs}
Let
\begin{equation*}
L_0(z,\psi,w)=(i\omega_0 z,\Omega_0,0),
\qquad
X=\mathbb C\times \mathbb S^1\times \mathbb R.
\end{equation*}
For $\alpha=(a,b,k,c)\in \mathbb N_0\times \mathbb N_0\times \mathbb Z\times \mathbb N_0,$
define the Fourier-Taylor monomial
\begin{equation*}
M_\alpha(z,\psi,w):=z^a\bar z^b e^{ik\psi}w^c.
\end{equation*}
Let $e_z,e_\psi,e_w$ denote the coordinate vector fields in the $z$, $\psi$,
and $w$ directions. Then the vector-field monomials
\begin{equation*}
M_\alpha e_j,\qquad j\in{z,\psi,w},
\end{equation*}
are eigenvectors of the homological operator $\operatorname{ad}_{L_0}$ defined in
\eqref{eq:homological-operator}. More precisely,
\begin{equation*}
\operatorname{ad}_{L_0}(M_\alpha e_j)
=
i\lambda_\alpha^{(j)}M_\alpha e_j,
\end{equation*}
where
\begin{equation*}
\begin{aligned}
\lambda_\alpha^{(z)}
&=(a-b-1)\omega_0+k\Omega_0,\
\lambda_\alpha^{(\psi)}
=(a-b)\omega_0+k\Omega_0,\
\lambda_\alpha^{(w)}
=(a-b)\omega_0+k\Omega_0.
\end{aligned}
\end{equation*}
We call $\lambda_\alpha^{(j)}$ the homological eigenvalue associated with the
basis vector field $M_\alpha e_j$.
\end{lemma}

\begin{proof}
First, observe that the derivative of $M_\alpha$ along $L_0$ is
\begin{equation*}
D M_\alpha(u) \cdot L_0(u)
=
(i\omega_0 z)\partial_z M_\alpha
+(-i\omega_0\bar z)\partial_{\bar z}M_\alpha
+\Omega_0\partial_\psi M_\alpha .
\end{equation*}
Hence
\begin{equation*}
D M_\alpha(u)\cdot L_0(u)
=
i\big((a-b)\omega_0+k\Omega_0\big)M_\alpha(u).
\end{equation*}
Next, recall that
\begin{equation*}
\operatorname{ad}_{L_0}M_\alpha(u)e_j=[M_\alpha(u)e_j,L_0]
=
D M_\alpha(u)e_j \cdot L_0(u)-D L_0(u)\cdot M_\alpha(u)e_j(u).
\end{equation*}
We compute the two terms separately. The first term satisfies $D(M_\alpha e_j) \cdot {L_0}=(DM_\alpha \cdot {L_0})\,e_j$.
For $j\in\{\psi,w\}$, the corresponding components of $L_0$ are constant, so $D L_0 \cdot {M_\alpha e_\psi}=0$ and
$D L_0 \cdot {M_\alpha e_w}=0$. 
Therefore
\begin{equation*}
\operatorname{ad}_{L_0}(M_\alpha e_\psi)
=
i\big((a-b)\omega_0+k\Omega_0\big)M_\alpha e_\psi,
\end{equation*}
and
\begin{equation*}
\operatorname{ad}_{L_0}(M_\alpha e_w)
=
i\big((a-b)\omega_0+k\Omega_0\big)M_\alpha e_w.
\end{equation*}
For $j=z$, we use $(L_0)_z=i\omega_0z$. Thus
\begin{equation*}
D L_0(u)\cdot M_\alpha e_z=i\omega_0M_\alpha e_z.
\end{equation*}
Consequently,
\begin{equation*}
\operatorname{ad}_{L_0}(M_\alpha e_z)
=i\big((a-b)\omega_0+k\Omega_0\big)M_\alpha e_z - i\omega_0M_\alpha e_z,
\end{equation*}
which gives
\begin{equation*}
\operatorname{ad}_{L_0}(M_\alpha e_z)
=
i\big((a-b-1)\omega_0+k\Omega_0\big)M_\alpha e_z.
\end{equation*}
This proves the stated eigenvalue formulas.
\end{proof}

In the present paper, the order-$\varepsilon$ vector field $F$ is a finite
combination of polynomial terms in $z,\bar z,w$ and Fourier modes in $\psi$.
Thus each component of $F$ can be expanded as
\begin{equation}\label{eq:fourierExp}
F_j(u)=\sum_\alpha f_j^\alpha M_\alpha(u),
\quad
j\in\{z,\psi,w\}.
\end{equation}
Here, $f_j^\alpha\in\mathbb C$. Since the components $F_\psi$ and $F_w$
are real-valued, their coefficients satisfy the conjugate-symmetry relation
\begin{equation*}
f_j^{(b,a,-k,c)}
=
\overline{f_j^{(a,b,k,c)}},
\qquad
j\in\{\psi,w\}.
\end{equation*} The homological eigenvalues in
Lemma~\ref{lem:homological-eigenpairs} determine which Fourier--Taylor
monomials are resonant. Terms for which
$\lambda_\alpha^{(j)}\neq0$ are non-resonant and can be removed by a
near-identity transformation, whereas terms for which
$\lambda_\alpha^{(j)}=0$ are resonant and remain in the first-order normal
form. This motivates the following decomposition of $F$ into resonant and
non-resonant parts.

\begin{definition}
The resonant and non-resonant index sets are defined as
\begin{equation*}
\Lambda_j^{\mathcal{R}} := \{\alpha \in \mathcal A : \lambda^{(j)}_\alpha = 0\},
\qquad
\Lambda_j^{\mathcal{N}} := \{\alpha \in \mathcal A : \lambda^{(j)}_\alpha \neq 0\}
\end{equation*} for $j\in\{z,\psi,w\}.$
Accordingly, we decompose the $j$-th component of $F$, as given in \eqref{eq:fourierExp}, as
\begin{equation}\label{eq:resonant-splitting}
F_j
=
F^{\mathcal{R}}_j
+
F^{\mathcal{N}}_j,
\end{equation}
where
\begin{equation*}
F^{\mathcal{R}}_j
:=
\sum_{\alpha\in\Lambda_j^{\mathcal{R}}} f^{\alpha}_{j}\,M_\alpha,
\qquad
F^{\mathcal{N}}_j
:=
\sum_{\alpha\in\Lambda_j^{\mathcal{N}}} f^{\alpha}_{j}\,M_\alpha.
\end{equation*}
Here, $F^{\mathcal{R}}_j$ is called the resonant part of $F_j$ and $F^{\mathcal{N}}_j$ its non-resonant part.
\end{definition}
The following theorem is the main result of this section. It shows that, after a suitable near-identity transformation, the order-$\varepsilon$ part of the vector field is reduced exactly to its resonant part, while all non-resonant terms are eliminated. This normal form will serve as the starting point for the analysis in the next section.
\begin{theorem}\label{thm:first-order-normal-form}
Consider \eqref{eq:NF-system}, and suppose that each component $F_j$, $j\in{z,\psi,w}$, admits a finite Fourier-Taylor expansion of the form
\begin{equation}
F_j(u)=\sum_{\alpha\in\mathcal A} f_j^\alpha M_\alpha(u),
\qquad
j\in{z,\psi,w},
\end{equation}
where only finitely many coefficients $f_j^\alpha$ are nonzero. Then, there exists a near-identity transformation of the form
$\bar u= u+\varepsilon H(u) + \mathcal{O}(\varepsilon^2)$ such that under this transformation \eqref{eq:NF-system} becomes
\begin{equation}\label{eq:NF-first-order-result}
\dot{\bar u}
\;=\;
L_0\,(\bar u)
\;+\;
\varepsilon\,\bar F(\bar u)
\;+\;\mathcal{O}(\varepsilon^2),
\end{equation}
where $\bar F$ is the resonant projection of $F$, i.e.
\begin{equation}\label{eq:NF-resonant-projection}
\bar F_j(u)
\;=\;
\sum_{\alpha\in\Lambda_j^{\mathcal{R}}} f^{\alpha}_{j}\,M_\alpha(u),
\qquad
j\in\{z,\psi,w\}.
\end{equation}
\end{theorem}
\begin{proof}Consider the near-identity change of coordinates $\tilde u=u+\varepsilon H(u)$. Expanding the transformed vector field to first order in $\varepsilon$ gives
\begin{equation*}
\dot{\bar u}
=
L_0(\bar u)
+
\varepsilon\Bigl(F(\bar u)+[H,L_0](\bar u)\Bigr)
+
\mathcal O(\varepsilon^2).
\end{equation*}
To eliminate as many order-$\varepsilon$ terms as possible, we consider the homological equation
\begin{equation}\label{eq:homological-equation}
[L_0,H] \;=\; F.
\end{equation}
Assume that $H$ has the Fourier-Taylor expansion $H_j(u)=\sum_{\alpha\in\mathcal A} h_j^\alpha\,M_\alpha(u).$
Using the Fourier-Taylor expansions of $F$ and $H$, together with the fact that the homological operator acts diagonally on the basis vector fields $M_\alpha e_j$, equation \eqref{eq:homological-equation} reduces to
\begin{equation}\label{eq:homological-equation-coef}
i\lambda^{(j)}_\alpha\, h^\alpha_j = f^\alpha_j,
\qquad
j\in\{z,\psi,w\}.
\end{equation}
This equation can be solved only when $\lambda^{(j)}_\alpha\neq 0$. Therefore, the non-resonant terms, namely those with $\alpha\in\Lambda_j^{\mathcal{N}}$, can be removed by choosing
\begin{equation}\label{eq:H-coefficients}
h^\alpha_j
=
\frac{f^\alpha_j}{i\lambda^{(j)}_\alpha},
\qquad
\alpha\in\Lambda_j^{\mathcal{N}}, \quad j\in\{z,\psi,w\},
\end{equation}
while for the resonant indices $\alpha\in\Lambda_j^\mathcal{R}$ we set
\begin{equation*}
h^\alpha_j =0.
\end{equation*}
Hence, the transformed system takes the form
\begin{equation*}
\dot{\bar u}
=
L_0(\bar u)
+
\varepsilon\,\bar F(\bar u)
+
\mathcal O(\varepsilon^2),
\end{equation*}
which is exactly \eqref{eq:NF-first-order-result}. This proves the claim.
\end{proof}
From now on, we drop the bar and write $u$ instead of $\bar u$ for the normal-form coordinates.

\section{Normal Form Symmetry}\label{sec:NFSym}

In Section~4, we transformed the reduced oscillator-rotator system into its first-order normal form and obtained the commutation relation $[L_0,\bar F]=0.$ Now we explain the geometric meaning of this relation. The flow of the unperturbed vector field $L_0$ consists of rotation in the complex $z$-plane and translation in the solitary phase $\psi$. The commutation relation implies that the truncated normal form is equivariant under the action of this flow. 

The main result of this section is Theorem~\ref{thm:NF-equivariant}, where we make this statement precise by proving the equivariance of the truncated first-order normal form. This symmetry is the reason for introducing the normal form: it allows the periodic solutions considered later to be characterized as relative equilibria of the symmetry action. In Section~6, we use this equivariance to formulate the relative-equilibrium problem and reduce it to algebraic conditions.

Recall that we work on the phase space
\begin{equation*}
X \;=\; \mathbb{C}\times \mathbb{S}^1\times \mathbb{R},
\quad \text{with elements} \quad
u=(z,\psi,w).
\end{equation*}
The unperturbed vector field in \eqref{eq:NF-system} is given by \eqref{eq:L0-vector-field}. 
Its flow defines a one-parameter family of diffeomorphisms $R_\theta:X\to X$,
$\theta\in\mathbb{R}$, of the form
\begin{equation}\label{eq:Rtheta}
R_\theta(z,\psi,w)
=
\bigl(e^{i\omega_0 \theta} z,\ \psi + \Omega_0 \theta,\ w\bigr).
\end{equation}
For $u\in X$, we denote the $R_\theta$-orbit of $u$ by
\begin{equation}\label{eq:OrbitRtheta}
\mathcal{O}(u)
=
\bigl\{\,R_\theta(u) : \theta \in \mathbb{R}\,\bigr\}\subset X. 
\end{equation}
We will make the following assumption from now on.
\begin{assumption}\label{ass:FreqRatio}
The frequency ratio $\frac{\omega_0}{\Omega_0} \in \mathbb{Q}$ is rational.
\end{assumption}
\begin{lemma}\label{lem:common-period}
Assume that Assumption~\ref{ass:FreqRatio} holds. Then there exists $T>0$ such that
\begin{equation}\label{eq:Rtheta-periodic}
R_{\theta+T}=R_\theta
\qquad\text{for all }\theta\in\mathbb{R}.
\end{equation}
In particular, $R_T=\mathrm{Id}_X$. In other words, $R_\theta$ defines a circle action.
\end{lemma}

\begin{proof}
Assume that Assumption~\ref{ass:FreqRatio} holds, so there exist coprime integers
$m,n\in\mathbb{Z}$ and $\omega^*>0$ such that $\omega_0=m\omega^*$ and $\Omega_0=n\omega^*$.
Define
\begin{equation}\label{eq:common-period-T}
T \;:=\; \frac{2\pi}{\omega^*}
\;=\;\frac{2\pi m}{\omega_0}
\;=\;\frac{2\pi n}{\Omega_0}.
\end{equation}
By \eqref{eq:Rtheta} we have
\begin{align*}
R_{\theta+T}(z,\psi,w)
&=
\bigl(e^{i\omega_0(\theta+T)} z,\ \psi + \Omega_0(\theta+T),\ w\bigr) \\
&=
\bigl(e^{i\omega_0\theta} e^{i2\pi m}z,\ \psi + \Omega_0\theta + 2\pi n,\ w\bigr) \\
&=
\bigl(e^{i\omega_0\theta} z,\ \psi + \Omega_0\theta,\ w\bigr)
=
R_\theta(z,\psi,w).
\end{align*}
This proves \eqref{eq:Rtheta-periodic}.
Taking $\theta=0$ yields $R_T=\mathrm{Id}_X$.
\end{proof}

\begin{definition}\label{def:Rtheta-equivariant}
A smooth vector field $F:X\to TX$ is called \emph{$R_\theta$-equivariant} if
\begin{equation}\label{eq:RthetaSymmetry}
D R_\theta(u)\,F(u) \;=\; F\bigl(R_\theta(u)\bigr)
\end{equation}
for all $u\in X$ and all $\theta\in\mathbb{R}$.
\end{definition}

The proofs of the following two lemmas are omitted, since they are standard consequences of the general theory of equivariant vector fields. For details, see \cite{lee_introduction_2012}, Theorems 9.42 and 9.44.
\begin{lemma}\label{lem:equivariance-vs-bracket}
Let $F \in C^1$, then $F$ is $R_\theta$-equivariant if and only if
\begin{equation}
[F,L_0] = 0.
\end{equation}
\end{lemma}

\begin{lemma}\label{lem:FlowCommutes}
Let $F$ be an $R_\theta$-equivariant vector field on $X$ and let $\Phi_t$ denote its flow.
Then, for all $u\in X$, $t\in\mathbb{R}$ and $\theta\in\mathbb{R}$,
\begin{equation}\label{eq:FlowCommutes}
\Phi_t\bigl(R_\theta(u)\bigr)
=
R_\theta\bigl(\Phi_t(u)\bigr).
\end{equation}
\end{lemma}
The following theorem is the main result of this section. It shows that the truncated first-order normal form inherits the $R_\theta$-equivariance generated by $L_0$.
\begin{theorem}\label{thm:NF-equivariant}
The truncated first-order normal form
\begin{equation}\label{eq:NF-truncated}
\dot u \;=\mathcal{F} \;=\; L_0(u)+\varepsilon\,\bar F(u)
\end{equation}
is $R_\theta$-equivariant.
\end{theorem}
\begin{proof}
By construction, see Theorem~\ref{thm:first-order-normal-form},
the order-$\varepsilon$ term $\bar F$ of the normal form
lies in the kernel of $\mathrm{ad}_{L_0}$.
Hence
\begin{equation}
[L_0,\bar F]=0.
\end{equation}
Since $[L_0,L_0]=0$, we obtain
\begin{equation}
[L_0,\,L_0+\varepsilon\,\bar F]
=
[L_0,L_0]
+
\varepsilon\,[L_0,\bar F]
=
0.
\end{equation}
By Lemma~\ref{lem:equivariance-vs-bracket},
a vector field commutes with $L_0$ if and only if it is
$R_\theta$-equivariant. Therefore,
the truncated normal form \eqref{eq:NF-truncated}
is $R_\theta$-equivariant.
\end{proof}

\begin{remark}
By Lemma~\ref{lem:FlowCommutes}, the symmetry $R_\theta$
maps solutions of \eqref{eq:NF-truncated} to solutions of \eqref{eq:NF-truncated}.
In particular, if $u(t)$ is a solution of \eqref{eq:NF-truncated},
then $R_\theta(u(t))$ is also a solution for every $\theta\in\mathbb{R}$.
\end{remark}

\begin{remark}
    In components, \eqref{eq:RthetaSymmetry} is equivalent to
\begin{align}
e^{i\omega_0\theta} F_z(z,\psi,w)
&=
F_z\bigl(e^{i\omega_0\theta}z,\ \psi+\Omega_0\theta,\ w\bigr),
\label{eq:RthetaSymmetry-z}\\
F_\psi(z,\psi,w)
&=
F_\psi\bigl(e^{i\omega_0\theta}z,\ \psi+\Omega_0\theta,\ w\bigr),
\label{eq:RthetaSymmetry-psi}\\
F_w(z,\psi,w)
&=
F_w\bigl(e^{i\omega_0\theta}z,\ \psi+\Omega_0\theta,\ w\bigr).
\label{eq:RthetaSymmetry-w}
\end{align}
Thus, the $z$-component of $F$ is equivariant under the action $R_\theta$, whereas the $\psi$- and $w$-components are invariant under this action.\end{remark}  The following lemma shows that, for the Fourier-Taylor basis vector fields
$M_\alpha e_j$, the equivariance condition is equivalent to the resonance
condition $\lambda_\alpha^{(j)}=0$.
\begin{lemma}\label{lem:resonant-equivariant-monomials}
For each $\alpha\in\mathcal A$ and $j\in\{z,\psi,w\}$, the Fourier-Taylor basis
vector field $M_\alpha e_j$ is $R_\theta$-equivariant if and only if
\begin{equation}
\lambda_\alpha^{(j)}=0.
\end{equation}
Equivalently, the resonant basis vector fields are precisely the
$R_\theta$-equivariant basis vector fields.
\end{lemma}

\begin{proof}
Recall that
\begin{equation*}
R_\theta(z,\psi,w)
=
\bigl(e^{i\omega_0\theta}z,\ \psi+\Omega_0\theta,\ w\bigr).
\end{equation*}
Hence, for any $\alpha$ we have
\begin{align*}
M_\alpha\bigl(R_\theta(u)\bigr)
&=
\bigl(e^{i\omega_0\theta}z\bigr)^a
\bigl(e^{-i\omega_0\theta}\bar z\bigr)^b
e^{ik(\psi+\Omega_0\theta)}w^c \notag\\
&=
e^{i((a-b)\omega_0+k\Omega_0)\theta}\,M_\alpha(u).
\end{align*}
On the other hand,
\begin{equation*}
D R_\theta(u)e_z=e^{i\omega_0\theta}e_z,
\qquad
D R_\theta(u)e_\psi=e_\psi,
\qquad
D R_\theta(u)e_w=e_w.
\end{equation*}
For $j=z$, we therefore obtain
\begin{equation*}
D R_\theta(u)\bigl(M_\alpha(u)e_z\bigr)
=
e^{i\omega_0\theta}M_\alpha(u)e_z,
\end{equation*}
whereas
\begin{equation*}
(M_\alpha e_z)\bigl(R_\theta(u)\bigr)
=
M_\alpha\bigl(R_\theta(u)\bigr)e_z
=
e^{i((a-b)\omega_0+k\Omega_0)\theta}M_\alpha(u)e_z.
\end{equation*}
Thus $M_\alpha e_z$ is $R_\theta$-equivariant if and only if
\begin{equation*}
e^{i\omega_0\theta}
=
e^{i((a-b)\omega_0+k\Omega_0)\theta}
\qquad\text{for all }\theta,
\end{equation*}
which is equivalent to
\begin{equation*}
(a-b-1)\omega_0+k\Omega_0=0.
\end{equation*}
This is the case if and only if
\begin{equation*}
\lambda_\alpha^{(z)}=0.
\end{equation*}
For $j\in \{\psi,w\}$, we have
\begin{equation*}
M_\alpha(u) e_j
=
e^{i((a-b)\omega_0+k\Omega_0)\theta}M_\alpha(u)e_j.
\end{equation*}
Hence $M_\alpha e_\psi$ and $M_\alpha e_w$ are $R_\theta$-equivariant if and only if
\begin{equation*}
(a-b)\omega_0+k\Omega_0=0,
\end{equation*}
that is, if and only if
\begin{equation*}
\lambda_\alpha^{(j)}=0\qquad j \in\{ \psi,w\}.
\end{equation*}
This proves the claim.
\end{proof}

\section{Relative Equilibria}\label{sec:RelEq}
In the previous section, we showed that the flow of the unperturbed vector field $L_0$ defines the action $R_\theta$, and the truncated normal form is symmetric under this action. In this section, we exploit the $R_\theta$-symmetry to reduce the search for periodic solutions
of interest, namely resonant solitary states, to algebraic conditions imposed at a single point on an $R_\theta$-orbit. The main references for the theory of relative equilibria are \cite{krupa_bifurcations_1990,chossat_methods_2000,marsden_introduction_1999}. For background on symmetric
dynamics, we refer to 
\cite{field_lectures_2020,golubitsky_symmetry_2002}. For general background on Lie groups, we refer to
\cite{duistermaat_lie_2000}.

\subsection{Fundamental vector fields and tangency characterization}
We first recall the general language of relative equilibria for equivariant vector fields. This material is standard, but it is included to make clear why symmetry turns a periodic-orbit problem into an algebraic tangency condition. The key point is that a relative equilibrium is a solution whose time evolution is completely generated by the group action. Therefore it can be detected by checking equality between the vector field and a fundamental vector field at a single point.

Let $\mathcal G$ be a Lie group with Lie algebra $\mathfrak{g}$, acting smoothly on the
phase space $X$ by a left action. For each $\xi \in \mathfrak{g}$, the exponential map $\exp : \mathfrak{g} \to \mathcal G$
defines a one-parameter subgroup $t \mapsto \exp(t\xi) \in \mathcal G$, whose action on
$X$ produces a curve
\begin{equation*}\label{eq:CurveFromExp}
t \mapsto \exp(t\xi)\cdot u \in X.
\end{equation*}

\begin{definition}\label{def:FundamentalField}
For $\xi \in \mathfrak{g}$, the corresponding \emph{fundamental vector field}
(or \emph{infinitesimal generator}) $\xi_X$ on $X$ is defined by
\begin{equation}\label{eq:FundamentalFieldDef}
\xi_X(u)
:=
\left.\frac{d}{dt}\right|_{t=0}
\bigl(\exp(t\xi)\cdot u\bigr),
\qquad u \in X.
\end{equation}
\end{definition}
Let $G$ be a $\mathcal G$-equivariant vector field on $X$ with flow $\Phi^G_t$.

\begin{definition}\label{def:RelEq-general}
A solution $u(t) = \Phi^G_t(u^*)$ of the differential equation
\begin{equation*}
\dot u = G(u)
\end{equation*}
is called a \emph{relative equilibrium} if there exists an element
$\xi \in \mathfrak{g}$ such that
\begin{equation*}\label{eq:RelEq-def}
\Phi^G_t(u^*) = \exp(t\xi)\cdot u^*
\qquad \text{for all } t \in \mathbb{R}.
\end{equation*}
The element $\xi$ is called the \emph{group velocity} of the relative equilibrium.
\end{definition}
Thus, the evolution of $u^*$ under the flow of $G$ coincides with a group
orbit of the one-parameter subgroup
\begin{equation*}
t \mapsto \exp(t\xi) \in \mathcal G
\end{equation*}
generated by $\xi$, and the trajectory of $u^*$ is obtained by moving along the group
orbit.

The following result gives a standard tangency characterization of relative
equilibria; see, for example, \cite{marsden_introduction_1999}.

\begin{proposition} \label{prop:RelEqFundamental}
Let $G$ be a $\mathcal G$-equivariant vector field on $X$ with flow $\Phi^G_t$. For
$u^* \in X$ and $\xi \in \mathfrak{g}$, the following are equivalent:
\begin{enumerate}
    \item $u^*$ is a relative equilibrium with velocity $\xi$, that is
    \begin{equation}\label{eq:RelEq-orbit}
    \Phi^G_t(u^*)
    =
    \exp(t\xi)\cdot u^*
    \qquad \text{for all } t \in \mathbb{R}.
    \end{equation}
    \item The vector field at $u^*$ coincides with the fundamental vector field
    generated by $\xi$:
    \begin{equation}\label{eq:RelEq-tangent}
    G(u^*) = \xi_X(u^*).
    \end{equation}
\end{enumerate}
\end{proposition}

\subsection{$\mathbb{S}^1$-symmetry of our system}
Recall that, under Assumption~\ref{ass:FreqRatio}, the unperturbed flow
$R_\theta$ in \eqref{eq:Rtheta} is periodic. Let $T>0$ denote a minimal period, so that $R_{\theta+T}=R_\theta$
for all $\theta\in\mathbb R$.
We therefore regard the group as parametrized by $\theta\in \mathbb R/T\mathbb Z$ and identify the symmetry group with $\mathcal G=\mathbb R/T\mathbb Z\cong \mathbb S^1.$ The action of $\mathcal G$ on
$X=\mathbb C\times \mathbb S^1\times \mathbb R$ is given by
\begin{equation*}
[\theta]\cdot u = R_\theta(u),
\qquad
u\in X,
\end{equation*}
where $[\theta]$ denotes the equivalence class of $\theta$ modulo $T$. For
$\xi\in\mathbb R$, the corresponding one-parameter subgroup is
\begin{equation*}
\exp(t\xi)=[t\xi]\in \mathbb R/T\mathbb Z,
\end{equation*}
and its action on $X$ is
\begin{equation*}
\exp(t\xi)\cdot u = R_{t\xi}(u).
\end{equation*}
The infinitesimal generator of the $\mathbb{S}^1$-action $R_\theta$ is the
vector field $L_0$ introduced in \eqref{eq:L0-vector-field}.
Thus, for $\xi \in \mathbb{R}$, the associated fundamental vector field is
\begin{equation}\label{eq:FundamentalFieldXi}
\xi_X(u)
=
\left.\frac{d}{dt}\right|_{t=0} R_{t\xi}(u)
=
\xi\,\frac{d}{d\theta} R_\theta(u)\Big|_{\theta=0}
=
\xi\,L_0(u).
\end{equation}

Combining Proposition~\ref{prop:RelEqFundamental} with
\eqref{eq:FundamentalFieldXi} yields the following characterization of relative equilibria for our
system.

\begin{lemma}\label{lem:RelEqOurSystem}
Let $\Phi^\mathcal{F}_t$ be the flow of 
$R_\theta$-equivariant vector field defined in \eqref{eq:NF-truncated}. A point $u^* \in X$ is a relative equilibrium with velocity
$\xi \in \mathbb{R}$ if and only if
\begin{equation}\label{eq:RelEqOurSystemCond}
\mathcal{F}(u^*) = \xi\,L_0(u^*).
\end{equation}
\end{lemma}

\begin{proof}
By Definition~\ref{def:RelEq-general}, $u^*$ is a relative equilibrium with
velocity $\xi$ if and only if
\begin{equation*}
\Phi^\mathcal{F}_t(u^*) = \exp(t\xi)\cdot u^* = R_{t\xi}(u^*)
\qquad \text{for all } t \in \mathbb{R}.
\end{equation*}
By Proposition~\ref{prop:RelEqFundamental}, this is equivalent to
\begin{equation*}
\mathcal{F}(u^*) = \xi_X(u^*).
\end{equation*}
Using \eqref{eq:FundamentalFieldXi}, we have
\begin{equation*}
\mathcal{F}(u^*) = \xi_X(u^*) = \xi\,L_0(u^*),
\end{equation*}
which proves \eqref{eq:RelEqOurSystemCond}.
\end{proof}

In particular, for an $S^1$-symmetry, a relative equilibrium with
nonzero group velocity is a periodic orbit. Indeed, if $u^*$ is a relative
equilibrium with group velocity $\xi\neq 0$, then
\begin{equation*}
u(t)=\Phi_t^\mathcal{F}(u^*)=R_{\xi t}(u^*).
\end{equation*}
Since $R_\theta$ is periodic in $\theta$ with period $T$, we have
\[
u(t+T/\xi)=u(t).
\]
Thus, any relative equilibrium gives a periodic orbit of the system.  
\subsection{Characterization of Relative Equilibria on a Cross-Section}
The relative-equilibrium condition still contains the redundancy coming from the $S^1$-symmetry: every point on the same group orbit represents the same periodic solution. To remove this redundancy, we introduce a cross-section that selects a representative of each nontrivial orbit by fixing the phase of $z$.

\begin{lemma}\label{lem:CrossSection}
Define the cross-section
\begin{equation}\label{eq:CrossSection}
\Sigma
=
\bigl\{(z,\psi,w)\in X : \Im (z) = 0,\ \Re (z)>0\bigr\}.
\end{equation}
Let $u=(z,\psi,w)\in X$ with $z\neq 0$. Then there exists a point
$u^\Sigma \in \Sigma$ on its $\mathbb{S}^1$-orbit under $R_\theta$, i.e.\ there
is a $\theta\in\mathbb{R}$ such that
\begin{equation*}
u^\Sigma := R_\theta(u) \in \Sigma.
\end{equation*}
Furthermore, if $\omega_0= \Omega_0$ this $u^\Sigma$ is unique.
\end{lemma}

\begin{proof}
Write $z = r e^{i\varphi}$ with $r>0$ and $\varphi\in\mathbb{R}$. The
$z$ component of $R_\theta(u)$ in \eqref{eq:Rtheta} is
\begin{equation*}
e^{i\omega_0\theta}z = r e^{i(\varphi+\omega_0\theta)}.
\end{equation*}
Choosing $\theta = -\varphi/\omega_0$ modulo $2\pi/\omega_0$ makes this real
and positive, so $R_\theta(u)\in\Sigma$, which proves existence.

For uniqueness, suppose $\omega_0= \Omega_0$ and $R_{\theta_1}(u),R_{\theta_2}(u)\in\Sigma$. Then
$e^{i\omega_0\theta_1}z$ and $e^{i\omega_0\theta_2}z$ are both positive real
numbers. With $z = r e^{i\varphi}$ this means
$e^{i(\varphi+\omega_0\theta_j)} = 1$ for $j=1,2$, so
$\varphi+\omega_0\theta_j \in 2\pi\mathbb{Z}$. Subtracting yields
$\omega_0(\theta_1-\theta_2) \in 2\pi\mathbb{Z}$, and since $R_\theta$ is
$2\pi/\omega_0$-periodic in $\theta$ we obtain $R_{\theta_1}(u)=R_{\theta_2}(u)$.
Thus $u^\Sigma$ is unique.
\end{proof}

\begin{lemma}\label{lem:RelEqOrbit}
Let $u^*\in X$ be a relative equilibrium of the truncated system
\eqref{eq:NF-truncated} with velocity $\xi$, i.e.
\begin{equation}\label{eq:RelEqOrbitFlow}
\Phi^\mathcal{F}_t(u^*) = R_{\xi t}(u^*)
\qquad\text{for all } t\in\mathbb{R}.
\end{equation}
For some $\theta\in\mathbb{R}$, let $u^\Sigma := R_\theta(u^*) \in \Sigma$ be a representative in $\Sigma$. Then
$u^\Sigma$ is also a relative equilibrium with the same velocity $\xi$:
\begin{equation}\label{eq:RelEqOrbitFlowSigma}
\Phi_t(u^\Sigma) = R_{\xi t}(u^\Sigma)
\qquad\text{for all } t\in\mathbb{R}.
\end{equation}
\end{lemma}

\begin{proof}
Using Lemma \ref{lem:FlowCommutes}, we can write
\begin{equation*}
\Phi_t(u^\Sigma)
=
\Phi_t(R_\theta(u^*))
=
R_\theta(\Phi_t(u^*))
=
R_\theta(R_{\xi t}(u^*))
=
R_{\xi t}(R_\theta(u^*))
=
R_{\xi t}(u^\Sigma),
\end{equation*}
which shows \eqref{eq:RelEqOrbitFlowSigma}.
\end{proof}
We now combine Lemmas \ref{lem:RelEqOurSystem} and \ref{lem:RelEqOrbit} to characterize relative equilibria on the
cross-section $\Sigma$. The following theorem rewrites the
relative-equilibrium condition as three algebraic equations, which will be
solved for the Kuramoto model in Section~7.

\begin{theorem}\label{thm:RelEqCross}
Let $\mathcal{F}(u) = L_0(u) + \varepsilon  \bar F(u)$ be the truncated
$R_\theta$-equivariant vector field in \eqref{eq:NF-truncated}, and let
$u^*=(z^*,\psi^*,w^*)$ be a relative equilibrium of this system with $z^*\neq 0$.
Let $u^\Sigma\in\Sigma$ be a representative given by
Lemma~\ref{lem:CrossSection}. In particular, $u^\Sigma$ can be written as
\begin{equation}\label{eq:CrossRepForm}
u^\Sigma = (x^*,\psi^*,w^*),\qquad x^* \in \mathbb{R}^+.
\end{equation}
Then, $u^*$ has the group velocity
\begin{equation}\label{eq:RelEqVelocityC}
\xi
=
1+\frac{\varepsilon}{\Omega_0}\,
\bar F_\psi(x^*,\psi^*,w^*)
=
1+\frac{\varepsilon}{\omega_0 x^*}\,
\Im(\bar F_z(x^*,\psi^*,w^*)),
\end{equation}
and the following conditions hold:
\begin{align}
\Re(\bar F_z(x^*,\psi^*,w^*))
&= 0, \label{eq:RelEqCross1}\\[0.3em]
\bar F_w(x^*,\psi^*,w^*)
&= 0, \label{eq:RelEqCross2}\\[0.3em]
\frac{\Im(\bar F_z(x^*,\psi^*,w^*))}{x^*}
&=
\frac{\omega_0}{\Omega_0}\,
\bar F_\psi(x^*,\psi^*,w^*).
\label{eq:RelEqCross3}
\end{align}
Conversely, let $u^\Sigma=(x^*,\psi^*,w^*)\in\Sigma$ with $x^*>0$ satisfy
\eqref{eq:RelEqCross1}-\eqref{eq:RelEqCross3}. Then $u^\Sigma$ is a relative
equilibrium of the truncated system $\dot u = \mathcal{F}(u)$, with group velocity in \eqref{eq:RelEqVelocityC}  and every point on its
$\mathbb{S}^1$-orbit
\begin{equation*}
\mathcal{O}(u^\Sigma)
=
\{ R_\theta(u^\Sigma) : \theta\in\mathbb{S}^1 \}
\end{equation*}
is also a relative equilibrium. Equivalently,  any $u_0\in\mathcal{O}(u^\Sigma)$
 satisfies
\begin{equation}\label{eq:RelEqCrossSolution}
\Phi^\mathcal{F}_t(u_0) = R_{\xi t}(u_0),
\qquad t\in\mathbb{R}.
\end{equation}
\end{theorem}
\begin{proof}
\emph{First direction.}
Let $u^*$ be a relative equilibrium with $z^*\neq 0$, and let $u^\Sigma\in\Sigma$
be its unique representative as in Lemma~\ref{lem:CrossSection}. By
Lemma~\ref{lem:RelEqOrbit}, $u^\Sigma$ is again a relative equilibrium (with
the same velocity $\xi\in\mathbb{R}$). Lemma~\ref{lem:RelEqOurSystem} then yields
\begin{equation}\label{eq:RelEqCrossStart}
\mathcal{F}(u^\Sigma) = \xi\,L_0(u^\Sigma).
\end{equation}
Write $u^\Sigma=(x^*,\psi^*,w^*)$ with $x^*>0$. Using \eqref{eq:NF-truncated}, the $\psi$-component of \eqref{eq:RelEqCrossStart} gives
\begin{equation*}
\Omega_0+\varepsilon \bar F_\psi(x^*,\psi^*,w^*)
=
\xi\Omega_0,
\end{equation*}
and hence
\begin{equation*}
\xi
=
1+\frac{\varepsilon}{\Omega_0}
\bar F_\psi(x^*,\psi^*,w^*).
\end{equation*}
Similarly, taking the imaginary part of the $z$-component gives
\begin{equation*}
\omega_0x^*
+
\varepsilon\Im\!\left(
\bar F_z(x^*,\psi^*,w^*)
\right)
=
\xi\omega_0x^*,
\end{equation*}
so that
\begin{equation*}
\xi
=
1+
\frac{\varepsilon}{\omega_0x^*}
\Im\!\left(
\bar F_z(x^*,\psi^*,w^*)
\right).
\end{equation*}
This proves \eqref{eq:RelEqVelocityC}. Comparing the remaining components and the real
part of the $z$-component yields \eqref{eq:RelEqCross1}-\eqref{eq:RelEqCross3}. \\

\medskip
\emph{Second direction: } Let $u^\Sigma=(x^*,\psi^*,w^*)\in\Sigma$ with $x^*>0$ satisfy
\eqref{eq:RelEqCross1}-\eqref{eq:RelEqCross3}.
Define
\begin{equation}\label{eq:RelEqXiDef}
\xi := 1 + \frac{\varepsilon}{\Omega_0}\,
\bar F_\psi(x^*,\psi^*,w^*).
\end{equation}
We claim that
\begin{equation}\label{eq:RelEqCrossTangency}
\mathcal{F}(u^\Sigma) = \xi\,L_0(u^\Sigma).
\end{equation}
Indeed, by \eqref{eq:NF-truncated} and \eqref{eq:L0-vector-field} the $\psi$-component
of \eqref{eq:RelEqCrossTangency}  reduces exactly to \eqref{eq:RelEqXiDef}, so it
holds by definition of $\xi$. For the $w$-component, since $(L_0)_w=0$, condition
\eqref{eq:RelEqCross2} gives
$\mathcal F_w(u^\Sigma)=\xi(L_0)_w(u^\Sigma).$
Finally, using
\eqref{eq:RelEqCross1}-\eqref{eq:RelEqCross3} in the $z$ component yields the
remaining equality, so \eqref{eq:RelEqCrossTangency} holds.
By Lemma~\ref{lem:RelEqOurSystem}, \eqref{eq:RelEqCrossTangency} is equivalent to
$u^\Sigma$ being a relative equilibrium with velocity $\xi$. Lemma~\ref{lem:RelEqOrbit}
then implies that every point on the $\mathbb{S}^1$-orbit $\mathcal{O}(u^\Sigma)$
is also a relative equilibrium, and that the corresponding trajectories have the
form \eqref{eq:RelEqCrossSolution}. This proves the theorem.
\end{proof}
\begin{remark}
Condition~\eqref{eq:RelEqCross3} can be interpreted as a frequency-locking condition between the rotation of the complex variable $z$ and the rotation of the phase variable $\psi$.

Indeed, write $z=re^{i\varphi_z}$. At a relative equilibrium represented on the cross-section by $u^\Sigma=(x^*,\psi^*,w^*)$, condition~\eqref{eq:RelEqCross1} implies that the radial component vanishes. Hence the $z$-component has only angular motion, with angular velocity
$$
\dot\varphi_z
=
\omega_0
+
\varepsilon\frac{\Im \bar F_z(x^*,\psi^*,w^*)}{x^*}.
$$
On the other hand, the $(\psi)$-component has angular velocity
$$
\dot\psi
=
\Omega_0
+
\varepsilon \bar F_\psi(x^*,\psi^*,w^*).
$$
Using \eqref{eq:RelEqCross3}, we obtain
$$
\dot\varphi_z
=
\omega_0
+
\varepsilon\frac{\omega_0}{\Omega_0}
\bar F_\psi(x^*,\psi^*,w^*)
=
\frac{\omega_0}{\Omega_0}
\left(
\Omega_0
+
\varepsilon\bar F_\psi(x^*,\psi^*,w^*)
\right)
=
\frac{\omega_0}{\Omega_0}\dot\psi.
$$
Thus \eqref{eq:RelEqCross3} says that the corrected angular velocities of the $z$ and $\psi$ components have the same ratio as the unperturbed frequencies:
$$
\frac{\dot\varphi_z}{\dot\psi}
=
\frac{\omega_0}{\Omega_0}.
$$
In the $1{:}1$ resonant case, $\omega_0=\Omega_0,$ the two perturbed angular velocities are equal. Therefore the rotation in the $z$-plane and the rotation of the solitary phase $\psi$ have the same period.
\end{remark}

\section{Existence and Stability of Resonant Solitary States}\label{sec:Analysis}

We now apply the general construction of Sections~4-6 to the specific reduced system obtained from the Kuramoto model with inertia. In the previous sections, we showed that resonant periodic solutions can be found by solving the cross-section conditions in Theorem~\ref{thm:RelEqCross}. In this section, we investigate near $1{:}1$ resonance, see \eqref{eq:detuning}, between the selected cluster mode and the solitary oscillator.

The goal of this section is to compute what the general theory gives for this specific solitary-node problem. We first write the near-resonant normal form explicitly. We then solve the cross-section equations to obtain the amplitude, phase relation, and frequency correction of the corresponding relative equilibrium. Finally, we analyze stability, to determine whether the resulting resonant solitary state is orbitally stable.

We assume that the unperturbed frequencies are $\varepsilon$-close and write
\begin{equation}
\omega_0=\Omega_0+\varepsilon\Delta,
\label{eq:detuning}
\end{equation}
where $\Delta=\mathcal O(1)$ is a detuning parameter. Under this
assumption, the reduced system \eqref{eq:u} takes the form
\begin{equation}
\label{eq:near-resonant-reduced-system}
\begin{cases}
\displaystyle
\dot z
= i\Omega_0 z
+\varepsilon\,\bigg(i\Delta z-\dfrac{\hat\alpha}{2}\,(z-\bar z)-\dfrac{i\,C_{\ell}}{4\omega_0}\,(z+\bar z)^2
-\;i\,\dfrac{\kappa q}{\omega_0}\,\sin\psi\bigg)
+\mathcal{O}(\varepsilon^2),\\[2mm]
\dot\psi = \Omega_0 +\varepsilon w+\mathcal{O}(\varepsilon^2),\\[1mm]
\displaystyle
\dot w = \varepsilon\bigl(\hat P_N-\hat\alpha(\Omega+w)\bigr)+\mathcal{O}(\varepsilon^2).
\end{cases}
\end{equation}
We now apply Theorem \ref{thm:first-order-normal-form} from Section~\ref{sec:NF-transformations}. This theorem states that, up to an error of the order $\mathcal O(\varepsilon^2)$, \eqref{eq:near-resonant-reduced-system} is locally conjugate to its resonant projection, obtained by keeping only the order-$\varepsilon$ terms that commute with $L_0$.  To order $\mathcal{O}(\varepsilon)$ the normal form of  \eqref{eq:near-resonant-reduced-system} reads
\begin{equation}\label{eq:NForig}
\begin{cases}
\displaystyle
\dot z
= i\Omega_0 z
+\varepsilon\Big(
(i\Delta-\tfrac{\hat\alpha}{2})\,z
-\dfrac{\kappa q}{2\Omega_0}\, e^{i\psi}
\Big)
+\mathcal{O}(\varepsilon^2),\\[2mm]
\dot\psi = \Omega_0 +\varepsilon w +\mathcal{O}(\varepsilon^2),\\[1mm]
\displaystyle
\dot w = \varepsilon\bigl(\hat P_N-\hat\alpha(\Omega+w)\bigr)+\mathcal{O}(\varepsilon^2).
\end{cases}
\end{equation}\label{eq:near-resonant-nf-system}
We now apply the cross-section conditions from
Theorem~\ref{thm:RelEqCross} to  \eqref{eq:NForig}. This gives an algebraic criterion for the
existence of a nontrivial relative equilibrium, together with an explicit
formula for it.

\begin{theorem}\label{thm:RelEqNFOrig}
Consider the system \eqref{eq:NForig} after truncation at order
$\mathcal O(\varepsilon)$. This system admits a nontrivial relative
equilibrium with $z^*\neq 0$ if and only if
\begin{equation*}
\hat\alpha\neq 0,
\qquad
\kappa q\neq 0.
\end{equation*}
Under these conditions, the relative equilibrium has a representative in the
cross-section $\Sigma=\{(z,\psi,w):z=x\in\mathbb R^+\}$ by $u^\Sigma=(x^*,\psi^*,w^*)$, where
\begin{align}
w^*
&=
\frac{\hat P_N}{\hat\alpha}-\Omega ,
\label{eq:wstar}\\[0.5em]
x^*
&=
\left\lvert
\frac{\kappa q}
{\Omega_0
\sqrt{\hat\alpha^2+
4\left(\frac{\hat P_N}{\hat\alpha}-\Omega-\Delta\right)^2}}
\right\rvert ,
\label{eq:xstar}\\[0.5em]
\cos\psi^*
&=
-\operatorname{sgn}(\kappa q)\,
\frac{\hat\alpha}
{\sqrt{\hat\alpha^2+
4\left(\frac{\hat P_N}{\hat\alpha}-\Omega-\Delta\right)^2}},
\label{eq:psistar-cos}\\[0.5em]
\sin\psi^*
&=
-\operatorname{sgn}(\kappa q)\,
\frac{2\left(\frac{\hat P_N}{\hat\alpha}-\Omega-\Delta\right)}
{\sqrt{\hat\alpha^2+
4\left(\frac{\hat P_N}{\hat\alpha}-\Omega-\Delta\right)^2}}.
\label{eq:psistar-sin}
\end{align}
Its group velocity is
\begin{equation*}\nonumber
\xi^*
=
1+\frac{\varepsilon}{\Omega_0}w^*.
\label{eq:cstar-analysis}
\end{equation*}
The full relative equilibrium is the $S^1$-orbit of
$u^\Sigma$ generated by the symmetry action.
\end{theorem}
\begin{proof}
After truncation at order $\mathcal O(\varepsilon)$, the system
\eqref{eq:NForig} satisfies
\begin{equation*}
\bar F_z(z,\psi,w)
=
\left(i\Delta-\frac{\hat\alpha}{2}\right)z
-
\frac{\kappa q}{2\Omega_0}e^{i\psi},
\qquad
\bar F_\psi(z,\psi,w)=w,
\end{equation*}
and
\begin{equation*}
\bar F_w(z,\psi,w)
=
\hat P_N-\hat\alpha(\Omega+w).
\end{equation*}
Since the truncated system is $S^1$-symmetric, its relative
equilibria are characterized by the cross-section conditions of
Theorem~\ref{thm:RelEqCross}.

Let
\begin{equation*}
u^\Sigma=(x^*,\psi^*,w^*)\in\Sigma,
\qquad x^*>0.
\end{equation*}
Condition \eqref{eq:RelEqCross2} gives
\begin{equation*}
\hat P_N-\hat\alpha(\Omega+w^*)=0.
\end{equation*}
Thus, provided $\hat\alpha\neq0$, we obtain
\begin{equation*}
w^*
=
\frac{\hat P_N}{\hat\alpha}-\Omega .
\end{equation*}
Next, using $z=x^*\in\mathbb R^+$, the first cross-section condition
\eqref{eq:RelEqCross1} gives
\begin{equation*}
-\frac{\hat\alpha}{2}x^*
-
\frac{\kappa q}{2\Omega_0}\cos\psi^*
=
0.
\label{eq:proof-cos-condition}
\end{equation*}
The third cross-section condition \eqref{eq:RelEqCross3} gives
\begin{equation*}
\Delta
-
\frac{\kappa q}{2\Omega_0x^*}\sin\psi^*
=
w^*.
\label{eq:proof-sin-condition}
\end{equation*}
Equivalently,
\begin{align}
\cos\psi^*
&=
-\frac{\Omega_0\hat\alpha x^*}{\kappa q},
\label{eq:proof-cos}\\[0.4em]
\sin\psi^*
&=
-\frac{2\Omega_0(w^*-\Delta)x^*}{\kappa q}.
\label{eq:proof-sin}
\end{align}
These equations require $\kappa q\neq0$ in the nontrivial case
$x^*>0$.

Using the identity $\sin^2\psi^*+\cos^2\psi^*=1$ in
\eqref{eq:proof-cos}-\eqref{eq:proof-sin}, we get
\begin{equation*}
\left(\frac{\Omega_0\hat\alpha x^*}{\kappa q}\right)^2
+
\left(\frac{2\Omega_0(w^*-\Delta)x^*}{\kappa q}\right)^2
=
1.
\end{equation*}
Hence
\begin{equation*}
x^*
=
\left\lvert
\frac{\kappa q}
{\Omega_0\sqrt{\hat\alpha^2+4(w^*-\Delta)^2}}
\right\rvert .
\end{equation*}
Substituting
\begin{equation*}
w^*
=
\frac{\hat P_N}{\hat\alpha}-\Omega
\end{equation*}
gives
\begin{equation*}
x^*
=
\left\lvert
\frac{\kappa q}
{\Omega_0
\sqrt{\hat\alpha^2+
4\left(\frac{\hat P_N}{\hat\alpha}-\Omega-\Delta\right)^2}}
\right\rvert ,
\end{equation*}
which is \eqref{eq:xstar}.

Substituting this value of $x^*$ into
\eqref{eq:proof-cos}-\eqref{eq:proof-sin} gives
\begin{align*}
\cos\psi^*
&=
-\operatorname{sgn}(\kappa q)\,
\frac{\hat\alpha}
{\sqrt{\hat\alpha^2+
4\left(\frac{\hat P_N}{\hat\alpha}-\Omega-\Delta\right)^2}},
\\[0.4em]
\sin\psi^*
&=
-\operatorname{sgn}(\kappa q)\,
\frac{2\left(\frac{\hat P_N}{\hat\alpha}-\Omega-\Delta\right)}
{\sqrt{\hat\alpha^2+
4\left(\frac{\hat P_N}{\hat\alpha}-\Omega-\Delta\right)^2}},
\end{align*}
which are \eqref{eq:psistar-cos} and \eqref{eq:psistar-sin}. These two
relations determine $\psi^*$ modulo $2\pi$.

Conversely, suppose $\hat\alpha\neq0$ and $\kappa q\neq0$, and define
$x^*$, $\psi^*$, and $w^*$ by
\eqref{eq:xstar}, \eqref{eq:psistar-cos}, \eqref{eq:psistar-sin}, and
\eqref{eq:wstar}. Then the above computations can be read backwards:
\eqref{eq:RelEqCross2} holds by the definition of $w^*$, while
\eqref{eq:RelEqCross1} and \eqref{eq:RelEqCross3} hold by the definitions
of $x^*$ and $\psi^*$. Therefore the cross-section conditions of
Theorem~\ref{thm:RelEqCross} are satisfied, and $u^\Sigma$ is a relative
equilibrium of the truncated system.

Finally, since $\bar F_\psi(u^\Sigma)=w^*$, the group velocity is
\begin{equation*}
\xi
=
1+\frac{\varepsilon}{\Omega_0}\bar F_\psi(u^\Sigma)
=
1+\frac{\varepsilon}{\Omega_0}w^*.
\end{equation*}
The corresponding relative equilibrium in the full phase space is obtained
by applying the $S^1$-action to $u^\Sigma$.\end{proof}
The previous theorem gives existence and explicit formulas, but it does not yet determine whether the corresponding resonant solitary state is stable. 
\begin{proposition}
Assume that $\hat\alpha>0$. Then the nontrivial relative
equilibrium given in Theorem~\ref{thm:RelEqNFOrig} is transversely
asymptotically stable.
\end{proposition}
\begin{proof}
Let $u^\Sigma=(x^*,\psi^*,w^*)\in\Sigma$ be the representative of the relative equilibrium. By
Theorem~\ref{thm:RelEqNFOrig}, its group velocity is
\begin{equation}
\label{eq:velocity}
\xi
=
1+\frac{\varepsilon}{\Omega_0}w^*,
\end{equation}
and the corresponding relative-equilibrium solution is given by
\begin{equation*}
R_{\xi t}(u^\Sigma)
=
\left(
e^{i\Omega_0\xi t}x^*,
\psi^*+\Omega_0\xi t,
w^*
\right).
\end{equation*}

We introduce a perturbation in a frame moving with this group orbit by
writing
\begin{equation}
\label{eq:perturbation}
\begin{pmatrix}
z(t)\\
\psi(t)\\
w(t)
\end{pmatrix}
=
\begin{pmatrix}
e^{i\Omega_0\xi t}\left(x^*+u(t)+iv(t)\right)\\
\psi^*+\Omega_0\xi t+\sigma(t)\\
w^*+\zeta(t)
\end{pmatrix}.
\end{equation}
We differentiate \eqref{eq:perturbation}
and substitute \eqref{eq:perturbation} into \eqref{eq:NForig}, which gives
\begin{equation}\label{eq:zcomp_pert}
\dot u+i\dot v
=
\varepsilon
\left[
\left(
i(\Delta-w^*)-\frac{\hat\alpha}{2}
\right)
\left(x^*+u+iv\right)
-
\frac{\kappa q}{2\Omega_0}
e^{i\psi^*}e^{i\sigma}
\right]
\end{equation}
\begin{equation*}
\dot\sigma=\varepsilon\zeta,
\qquad
\dot\zeta=-\varepsilon\hat\alpha\zeta,
\end{equation*}
Since $u^\Sigma$ is a relative equilibrium, its $z$-component satisfies
\begin{equation}\label{eq:zcomp_equilibrium}
\left(
i(\Delta-w^*)-\frac{\hat\alpha}{2}
\right)x^*
-
\frac{\kappa q}{2\Omega_0}e^{i\psi^*}
=
0.
\end{equation}
Subtracting \eqref{eq:zcomp_equilibrium} from
\eqref{eq:zcomp_pert} yields
\begin{equation*}
\dot u+i\dot v
=
\varepsilon
\left[
\left(
i(\Delta-w^*)-\frac{\hat\alpha}{2}
\right)(u+iv)
-
\frac{\kappa q}{2\Omega_0}e^{i\psi^*}
\left(e^{i\sigma}-1\right)
\right].
\end{equation*}

We now linearize at $(u,v,\sigma,\zeta)=(0,0,0,0).$ Since $e^{i\sigma}-1
=
i\sigma+\mathcal O(\sigma^2),$
the  linearized $z$-equation becomes
\begin{equation*}\label{eq:zcomp_pert2}
\dot u+i\dot v
=
\varepsilon
\left[
\left(
i(\Delta-w^*)-\frac{\hat\alpha}{2}
\right)(u+iv)
-
i\frac{\kappa q}{2\Omega_0}e^{i\psi^*}\sigma
\right]
+
\mathcal O\left(
\varepsilon\|(u,v,\sigma,\zeta)\|^2
\right).
\end{equation*}
Using \eqref{eq:zcomp_equilibrium} again 
and separating real and imaginary parts yields
\begin{equation*}
\frac{d}{dt}
\begin{pmatrix}
u\\
v\\
\sigma\\
\zeta
\end{pmatrix}
=
\varepsilon
\begin{pmatrix}
-\hat\alpha/2 & w^*-\Delta & -(w^*-\Delta)x^* & 0\\
-(w^*-\Delta) & -\hat\alpha/2 & \hat\alpha x^*/2 & 0\\
0 & 0 & 0 & 1\\
0 & 0 & 0 & -\hat\alpha
\end{pmatrix}
\begin{pmatrix}
u\\
v\\
\sigma\\
\zeta
\end{pmatrix}
+
\mathcal O\left(
\varepsilon\|(u,v,\sigma,\zeta)\|^2
\right).
\end{equation*}
The eigenvalues of the Jacobian matrix are
\begin{equation*}
\lambda_{1,2}
=
\varepsilon
\left(
-\frac{\hat\alpha}{2}
\pm i(w^*-\Delta)
\right),
\qquad
\lambda_3=0,
\qquad
\lambda_4=-\varepsilon\hat\alpha.
\end{equation*}
For $\varepsilon>0$ and $\hat\alpha>0$, all eigenvalues transverse to the
group orbit have strictly negative real parts. Therefore, the relative
equilibrium is transversely asymptotically stable.
\end{proof}
\begin{remark}
    The relative equilibrium of the truncated normal form is transversely hyperbolic, since $\hat\alpha\neq 0$. The full normal-form system differs from the truncated system by higher-order terms of order $\mathcal O(\varepsilon^2)$. Therefore, for sufficiently small $\varepsilon$, these terms only perturb the location of the relative equilibrium and its transverse eigenvalues. Consequently, the corresponding nearby solution of the full system has the same transverse stability type. In particular, if $\hat\alpha>0$, the truncated normal form predicts a nearby orbitally stable resonant solitary state of the full system, whereas if $\hat\alpha<0$, the corresponding state is transversely unstable.

\end{remark}
\begin{figure}[t]
\centering
\includegraphics[width=\textwidth]{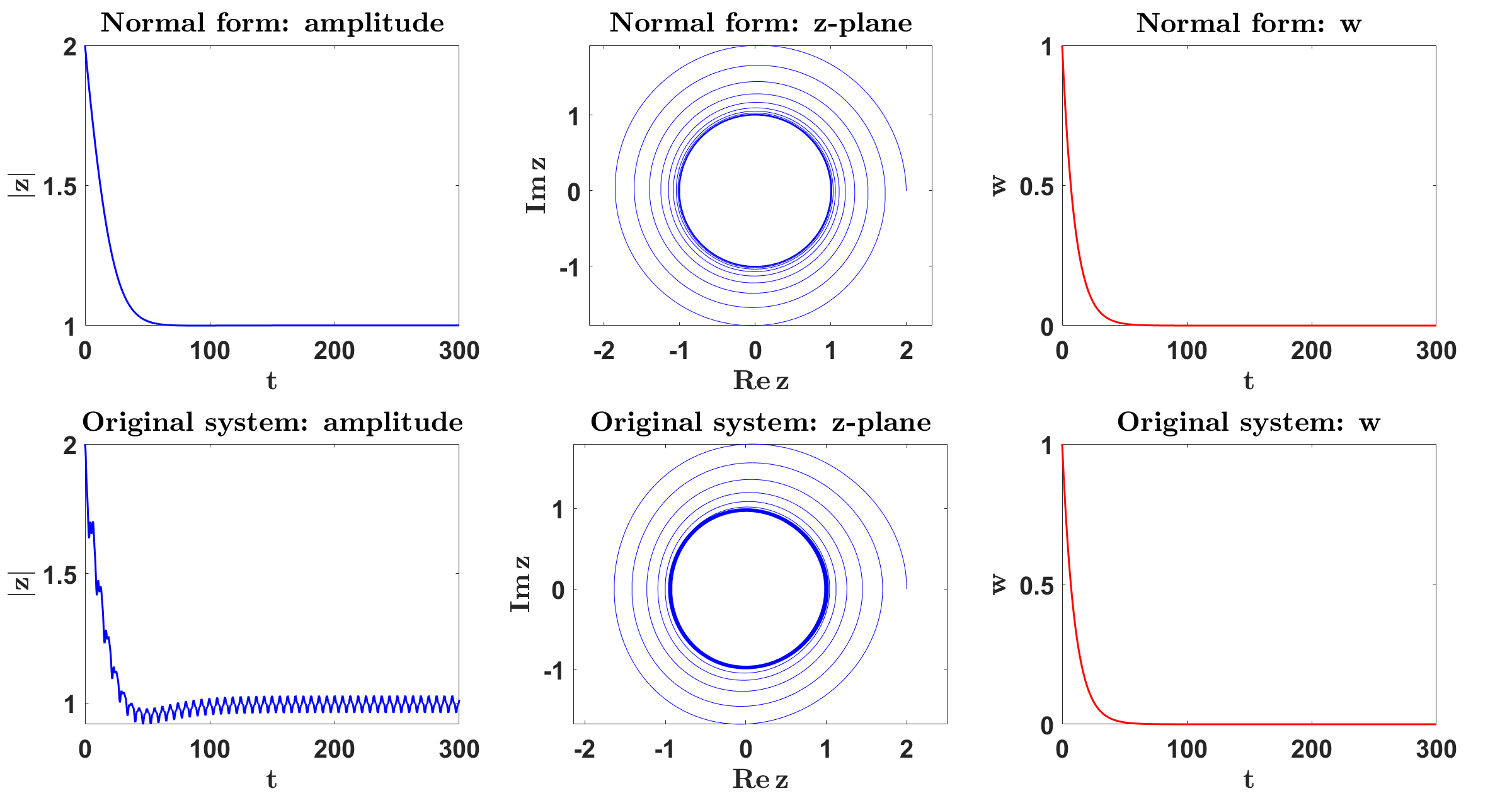}
\caption{Comparison of the truncated normal form in \eqref{eq:NForig} and the original (reduced) system in \eqref{eq:near-resonant-reduced-system} for $\varepsilon=0.1$, $\Omega_0=\Omega=1$, $\Delta=0$, $\hat{\alpha}=\hat P_N=q=\kappa=C_{\ell}=1$. The relative equilibrium satisfies $|z|=1$, $\psi=\pi$, and $w^*=0$. Top row: truncated normal form. Bottom row: original reduced system. Left column: evolution of the oscillator amplitude $|z|$. Middle column: trajectory in the complex $z$-plane. Right column: evolution of the solitary oscillator velocity $w$.}
\label{fig:numerics}
\end{figure}
Figure~\ref{fig:numerics} illustrates the stability of the
relative equilibrium for a representative choice of parameters. In both systems, the trajectory approaches the orbit associated with the resonant solitary state and $w$ converges to its equilibrium value. The truncated normal form \eqref{eq:NForig} exhibits monotone convergence of $|z|$ to its limiting value, whereas the original reduced system in \eqref{eq:near-resonant-reduced-system} shows small oscillatory corrections due to the non-resonant and higher-order terms neglected in \eqref{eq:NForig}.

\section{Conclusion and Discussion}\label{sec:Conclusion}

In this paper we introduced a normal-form and symmetry-based method for detecting periodic solutions in weakly perturbed oscillator-rotator systems. The main idea is to use normal form transformations to introduce symmetry, and reinterpret the relevant periodic solutions as relative equilibria of this symmetry action. This converts the search for periodic solutions into a finite-dimensional algebraic problem.

The method is general in the following sense. Once a system has been reduced to the form
$$
\dot u=L_0(u)+\varepsilon F(u)+\mathcal O(\varepsilon^2),
\qquad
u\in \mathbb C\times\mathbb S^1\times\mathbb R,$$
with an unperturbed oscillator-rotator flow generated by $L_0$, the resonant normal form selects the terms that commute with this flow. The resulting truncated system is equivariant under the corresponding $S^1$-action. Periodic solutions can then be found as relative equilibria, rather than by solving a self-consistency problem as in \cite{niehues_resonant_2024}. 

We applied this method to resonant solitary states in the Kuramoto model with inertia. Starting from a network with one solitary node attached to a synchronized cluster, we expanded around the synchronized cluster state, projected onto Laplacian eigenmodes, and retained a dominant resonant mode. This reduced the high-dimensional network dynamics to a harmonic oscillator coupled to a rotator. Applying the normal-form and relative-equilibrium method to this reduced system gave algebraic conditions for the existence of resonant solitary states.

For the near $1{:}1$ resonance, the relative-equilibrium conditions can be
solved explicitly. We determine the amplitude of the resonant
cluster mode, its phase relation with the solitary oscillator, and the
correction to the solitary frequency. Moreover, the normal form allows the
stability of this solution to be studied analytically. For positive damping, the relative equilibrium is
asymptotically stable modulo the $S^1$-symmetry.

Our method provides a different description of resonant solitary states from the
self-consistency approach in~\cite{niehues_resonant_2024}. In that work, resonant solitary frequencies are identified by requiring consistency between the periodic forcing generated by the solitary oscillator and the response of the synchronized cluster. Here, after reduction and resonant normal-form transformation, these states are identified as relative equilibria of an $S^1$-equivariant vector field. Their existence is therefore reduced to explicit algebraic conditions on a cross-section, while their transverse dynamics can be studied
directly from the reduced vector field. 

Several extensions of the present framework are natural. Here, we have considered the minimal configuration of a single solitary oscillator weakly coupled to one synchronized cluster and focused on a single resonant cluster mode. A first direction is to extend the reduction to networks containing several solitary oscillators or several synchronized clusters, possibly of different sizes and mean frequencies. Such configurations introduce several interacting rotator and oscillator degrees of freedom and may lead to higher-dimensional symmetry groups and more complicated resonance relations. This setting would also include collective states such as Cyclops states, which consist of two coherent clusters and a single solitary oscillator \cite{munyayev_cyclops_2023}, as well as more general multiclusters. It would be particularly interesting to determine whether these states can likewise be characterized as relative equilibria, or more generally as relative periodic solutions, of an appropriate resonant normal form. A second direction is to study how these solutions are created or destroyed as parameters vary. The normal-form formulation developed here provides a  starting point for analyzing bifurcations of resonant solitary and multicluster states and for determining how different solutions appear and are destroyed.

\section*{Acknowledgements}

We thank Jakob Niehues for many helpful discussions and insightful conversations related to this work.

\section*{Declaration on Generative AI Use}

Generative AI tools were used during the preparation of this manuscript
for language editing and editorial assistance. The authors reviewed and
verified the final manuscript and take full responsibility for its content.

\appendix
\section{Appendix: Reduction to Two Coupled Equations} \label{app:Reduction}
In this appendix we revisit the calculations presented in the supplementary material of \cite{niehues_resonant_2024}. Our overall procedure is similar, but instead of deriving the orders of magnitude of the parameters, we assume them from the start. We will start with the model given in \eqref{eq:ScaledModel}. Recall that for $\kappa_{kN}=0$, we assume that the uncoupled cluster dynamics in \eqref{eq:InitialSystem} admits a stable frequency synchronized solution \begin{equation}\label{eq:syncSol}
\phi_i(t)=\phi_i^*+\Omega t, \qquad i\in \mathcal S= \{ 1,2,\dots N-1\}.
\end{equation}
which satisfies the relation
\begin{equation} \label{eq:cluster-ansatz}
0
= P_i - \alpha\,\Omega
+ \sum_{j=1}^{N-1} K_{ij}\,\sin(\phi_j^*-\phi_i^*) \quad \text{for all}\quad  i \in \mathcal{S}.
\end{equation}

\subsection{Approximation near a solitary state} \label{approximation_near_solitary_state}
In this section, we derive the equations presented in Subsection~\ref{subsec:newcoordinates}. 
We perturb the cluster phases around the synchronized solution \eqref{eq:syncSol} by writing
\begin{equation}\label{app:clusterperturbed}
\phi_i(t)
=
\phi_i^*+\Omega t+\varepsilon \nu_i(t),
\qquad i\in\mathcal S,
\end{equation}
where $0<\varepsilon\ll 1 $ and we introduce the co-rotating solitary phase $\psi_N(t)$ by writing
\begin{equation}\label{app:solitarycorotating}
\phi_N(t)
=
\psi_N(t)+\phi_k^*+\Omega t.
\end{equation}
Thus $\psi_N$ is the phase of the solitary oscillator relative to the synchronized cluster, measured in the co-rotating frame of the attachment node $k$.

We now substitute \eqref{app:clusterperturbed} and \eqref{app:solitarycorotating} into the rescaled solitary-oscillator equation in \eqref{eq:ScaledModel}. Since
\begin{equation*}
\phi_k-\phi_N
=
\phi_k^*+\Omega t+\varepsilon\nu_k
-
\bigl(\psi_N+\phi_k^*+\Omega t\bigr)
=
\varepsilon\nu_k-\psi_N,
\end{equation*}
the coupling term \eqref{eq:ScaledModel} becomes
\begin{equation}\label{eq:SinAround}
\sin(\phi_k-\phi_N)
=
\sin(\varepsilon\nu_k-\psi_N)=-\sin(\psi_N) + \mathcal{O}(\varepsilon).
\end{equation}
Expanding \eqref{eq:ScaledModel} up to order $\varepsilon^2$, therefore gives the following equation for the solitary oscillator:
\begin{equation*}
\ddot{\psi}_N
= \varepsilon\,(\hat P_N-\hat\alpha \Omega
  - \,\hat\alpha\,\dot{\psi}_N)
  - \varepsilon^2\,\kappa_{Nk}\,\sin\psi_N
  + \mathcal{O}(\varepsilon^3).
\label{eq:psi-eq}
\end{equation*}
Similarly we insert \eqref{app:clusterperturbed} into the equation for the cluster in \eqref{eq:ScaledModel} and expand the sine terms to obtain
\begin{align}
\ddot{\nu}_i
&= \sum_{j=1}^{N-1}\kappa_{ij}\,
   \cos(\phi_j^*-\phi_i^*)(\nu_j-\nu_i)\notag\\
&\quad + \varepsilon\Bigg[
   -\hat\alpha\,\dot{\nu}_i
   - \frac{1}{2}\sum_{j=1}^{N-1}\kappa_{ij}\,
      \sin(\phi_j^*-\phi_i^*)(\nu_j-\nu_i)^2
   + \,\delta_{ik}\kappa_{kN}\,\sin\psi_N
   \Bigg]\notag\\
&\quad  + \mathcal{O}(\varepsilon^2),
\label{eq:nu}
\end{align}
where the second line follows from \eqref{eq:SinAround}. We define the graph Laplacian $L\in\mathbb{R}^{(N-1)\times(N-1)}$ as
\begin{equation*}
L_{ij} =
\begin{cases}
-\kappa_{ij}\,\cos(\phi_j^*-\phi_i^*), & i\neq j,\\[6pt]
\displaystyle\sum_{j\neq i}\kappa_{ij}\,\cos(\phi_j^*-\phi_i^*), & i=j,
\end{cases}
\label{eq:L-def}
\end{equation*}
and we introduce the notation
\begin{equation}
\bigl(N(\nu)\bigr)_i
:= - \frac{1}{2}\sum_{j=1}^{N-1}\kappa_{ij}\,
      \sin(\phi_j^*-\phi_i^*)(\nu_j-\nu_i)^2.
\label{eq:N-def}
\end{equation}
Then, \eqref{eq:nu} becomes
\begin{equation}\label{eq:nu2}
\ddot{\nu}_i + \sum_{j=1}^{N-1} L_{ij}\,\nu_j
= \varepsilon\Big(
      -\hat\alpha\,\dot{\nu}_i
      + \bigl(N(\nu)\bigr)_i
      + \delta_{ik}\kappa_{kN}\,\sin\psi_N\,
   \Big)
  + \mathcal{O}(\varepsilon^2).
\end{equation}
where $\bigl(N(\nu)\bigr)_i$ is given by \eqref{eq:N-def}.

\subsection{Laplacian Eigenmodes}
In this section, we rewrite equation \eqref{eq:nu2} in terms of the eigenmodes of the graph Laplacian.
Since the underlying graph is undirected, $L\in\mathbb{R}^{(N-1)\times(N-1)}$ is symmetric and hence orthogonally diagonalizable. That is,
\begin{equation*}
L = Q\Lambda Q^\top,
\end{equation*}
where $\Lambda=\mathrm{diag}(\lambda_1,\dots,\lambda_{N-1})$ and $Q\in\mathbb{R}^{(N-1)\times(N-1)}$ is orthogonal.
We define new coordinates $\eta$ as follows
\begin{equation*}
\eta := Q^\top \nu \in\mathbb{R}^{N-1},
\end{equation*} and we define coupling parameters by
\begin{equation*}
q^{(j)} := Q^\top e_j.
\end{equation*}
Since the columns of $Q$ are the Laplacian eigenvectors $v^{[i]}=Qe_i$, we have
\begin{equation*}
q^{(j)}_i = (Q^\top e_j)_i = e_j^\top Q e_i = Q_{ji} = v^{[i]}_j.
\end{equation*}
Equivalently,
\begin{equation*}
q^{(j)} = \bigl(v^{[1]}_j,\dots,v^{[N-1]}_j\bigr)^\top,
\end{equation*}
i.e., $q^{(j)}$ collects, for each mode $i$, the value of the $i$-th Laplacian eigenvector at node $j$.

The nonlinear term $N(\nu)$ is quadratic in $\nu$ and in terms of the new variable $\eta$ it can be written as
\begin{equation*}
\widetilde N(\eta) := Q^\top N(Q\eta),\quad \text{so}\quad 
\widetilde N_\ell(\eta) = \sum_{p,q=1}^{N-1} C_{\ell pq}\,\eta_p\eta_q,
\end{equation*}
for coefficients $C_{\ell pq}\in\mathbb{R}$ given by
\begin{equation*}
C_{\ell pq}
= -\frac{1}{2}
\sum_{i,j=1}^{N-1}\kappa_{ij}\,\sin(\phi_j^*-\phi_i^*)\,
Q_{i\ell}\,(Q_{jp}-Q_{ip})(Q_{jq}-Q_{iq}).
\end{equation*}
In summary, in terms of $\eta$, equation \eqref{eq:nu2} transforms into
\begin{equation} 
\ddot{\eta}_i + \lambda_i\,\eta_i
= \varepsilon\Big(
     -\hat\alpha\,\dot{\eta}_i
     + \widetilde N_i(\eta)
     + \kappa\,q^{(k)}_i\,\sin\psi_N
  \Big)
  + \mathcal{O}(\varepsilon^2).
\label{eq:eta-component}
\end{equation}
where we write $\kappa =\kappa_{kN}$ for brevity.

\subsection{Dominant-mode ansatz}

We fix an index $\ell\in\{1,\dots,N-1\}$ such that
\begin{equation*}
q^{(k)}_\ell \neq 0,
\end{equation*}
and assume that the cluster response is dominated by this mode, in the sense that
\begin{equation}
\eta_\ell(t)=\mathcal{O}(1),\qquad
\eta_p(t) =\mathcal{O}(\varepsilon)\quad\text{for all }p\neq\ell.
\label{eq:LocalizationAssumption}
\end{equation}
Thus $\eta(t) = \eta_\ell(t) e_\ell + \mathcal{O}(\varepsilon) $ and therefore
\begin{equation*}
\nu(t)=Q\eta(t)=\eta_\ell(t)\,Q e_{\ell}+\mathcal{O}(\varepsilon),
\end{equation*}
so the perturbation of the synchronized cluster is, to leading order, aligned with the single Laplacian eigenmode $v^{[\ell]}$.

Using \eqref{eq:LocalizationAssumption}, we can then write
\begin{equation*} \label{eq:Ntilde-dominant-mode}
\widetilde N_\ell(\eta)
= C_{\ell}\,\eta_\ell^2 + \mathcal{O}(\varepsilon),
\end{equation*}
where
\begin{equation*}
C_{\ell} :=C_{\ell \ell \ell} 
= -\frac{1}{2}
\sum_{i,j=1}^{N-1}\kappa_{ij}\,\sin(\phi_{j}^*-\phi_{i}^*)\,
Q_{i\ell}\,(Q_{j\ell}-Q_{i\ell})^2
\end{equation*}
is the coefficient associated with the nonlinear self-interaction of mode $\ell$. Moreover,
\begin{equation} \label{eq:q-dominant-mode}
q^{(k)}\cdot\eta = q\,\eta_\ell + \mathcal{O}(\varepsilon),
\qquad
q := q^{(k)}_\ell.
\end{equation}
Taking the $\ell$-th component of \eqref{eq:eta-component} and using \eqref{eq:Ntilde-dominant-mode} together with the notation introduced in \eqref{eq:q-dominant-mode}, we obtain

\begin{align}
&\ddot{\eta}_\ell + \lambda_\ell\,\eta_\ell
= \varepsilon\Big(
     -\hat\alpha\,\dot{\eta}_\ell
     + C_{\ell}\,\eta_\ell^2
     + \kappa\,q\,\sin\psi_N
  \Big)
  + \mathcal{O}(\varepsilon^2),
\label{eq:reduced-eta}\\
&\ddot{\psi}_N
= \varepsilon\,(\hat P_N-\hat\alpha\,\Omega
  - \hat\alpha\,\dot{\psi}_N)
  + \mathcal{O}(\varepsilon^2).
\label{eq:reduced-psi}
\end{align}
For notational convenience, we set
\begin{equation*}
\lambda_\ell = \omega_0^2,\qquad \omega_0>0,
\end{equation*}
so $\omega_0$ is the natural frequency of the dominant cluster mode.
Finally, we introduce the following coordinates:
\begin{equation*}
\eta := \eta_\ell,\qquad
v := \dot{\eta}_\ell,\qquad
\psi := \psi_N,\qquad
w := \dot{\psi}_N.
\label{eq:first-order-vars}
\end{equation*}
In this way, we remove the second time derivatives and obtain a first-order system on $\mathbb{R}^2 \times \mathbb{R}/\mathbb{Z} \times \mathbb{R} $. Specifically, the reduced dynamics \eqref{eq:reduced-eta}-\eqref{eq:reduced-psi} becomes
\begin{align}
\dot{\eta}
&= v,
\label{eq:fo-eta}\\
\dot{v}
&= -\omega_0^2\,\eta
   + \varepsilon\Big(
       -\hat\alpha\,v
       + C_{\ell}\,\eta^2
       + \kappa\,q\,\sin\psi
     \Big)
   + \mathcal{O}(\varepsilon^2),
\label{eq:fo-v}\\
\dot{\psi}
&= w,
\label{eq:fo-psi}\\
\dot{w}
&= \varepsilon\,(\hat P_N-\hat\alpha\,\Omega
   - \hat\alpha\,w)
   + \mathcal{O}(\varepsilon^2).
\label{eq:fo-w}
\end{align}
The coefficient $q=q^{(k)}_\ell$ measures the participation of the attachment node $k$ in the $\ell$-th Laplacian mode. In particular, the solitary oscillator can directly force mode $\ell$ only if $q\neq 0$. We assume from now on that this is the case, since otherwise both coupling terms involving $\psi_N$ vanish in \eqref{eq:eta-component} at the displayed orders, and the $\ell$-th mode would not be directly driven by the solitary oscillator via the attachment at node $k$.

\printbibliography
\end{document}